\documentclass[a4paper, reqno, 11pt]{amsart}
\usepackage[utf8]{inputenc}

\usepackage{amsthm}
\usepackage{amssymb}
\usepackage[american]{babel}
\usepackage{exscale}
\usepackage[mathscr]{euscript}
\usepackage{url}
\usepackage[all,ps,tips,tpic]{xy}
\usepackage{graphicx}
\usepackage{color}

\newtheorem{lem}{Lemma}[section]

\newtheorem{definition}[lem]{Definition}

\newtheorem{cor}[lem]{Corollary}
\newtheorem{thm}[lem]{Theorem}
\newtheorem*{thmA}{Theorem A}
\newtheorem*{thmB}{Theorem B}
\newtheorem*{thmC}{Theorem C}
\newtheorem*{thmD}{Theorem D}
\newtheorem{prop}[lem]{Proposition}

\theoremstyle{remark}
\newtheorem{rem}[lem]{Remark}

\DeclareMathOperator{\modd}{mod}
\DeclareMathOperator{\GL}{GL}

\newcommand{\tr}{\operatorname*{tr}}

\newdir{ >}{{}*!/-5pt/@{>}}
\SelectTips{cm}{10}

\begin{document}

\title{The Langlands-Kottwitz approach for the modular curve}
\author{Peter Scholze}
\begin{abstract}
We show how the Langlands-Kottwitz method can be used to determine the local factors of the Hasse-Weil zeta-function of the modular curve at places of bad reduction. On the way, we prove a conjecture of Haines and Kottwitz in this special case.
\end{abstract}

\date{\today}
\maketitle
\tableofcontents
\pagebreak

\section{Introduction}

The aim of this paper is to extend the method of Langlands, \cite{Langlands}, and Kottwitz, \cite{KottwitzTO}, to determine the Hasse-Weil zeta function of some moduli schemes of elliptic curves with level-structure, at all places. Fix a prime $p$ and an integer $m\geq 3$ prime to $p$. Let $\mathcal{M}_m/ \mathbb{Z}[\frac 1m]$ be the moduli space of elliptic curves with level-$m$-structure and let $\pi_n: \mathcal{M}_{\Gamma(p^n),m}\longrightarrow \mathcal{M}_m$ be the finite covering by the moduli space of elliptic curves with Drinfeld-level-$p^n$-structure and level-$m$-structure. Inverting $p$, this gives a finite Galois cover $\pi_{n\eta}: \mathcal{M}_{\Gamma(p^n),m}[\frac 1p]\cong \mathcal{M}_{p^nm}\longrightarrow \mathcal{M}_m[\frac 1p]$ with Galois group $\mathrm{GL}_2(\mathbb{Z}/p^n\mathbb{Z})$.

We make use of the concept of semisimple trace, \cite{Rapoport}, cf. also \cite{HainesNgo}, section 3.1. Recall that in the case of a proper smooth variety $X$ over $\mathbb{Q}$ with good reduction at $p$, the local factor of the Hasse-Weil zeta function is given by
\begin{equation}\label{eqngoodred}
\log \zeta_p(X,s) = \sum_{r\geq 1} |\mathfrak{X}(\mathbb{F}_{p^r})| \frac{p^{-rs}}{r}\ ,
\end{equation}
for any proper smooth model $\mathfrak{X}$ over $\mathbb{Z}_{(p)}$ of $X$. This follows from the proper base change theorem for \'{e}tale cohomology and the Lefschetz trace formula.

In general, for the semisimple local factor, $\zeta_p^{\mathrm{ss}}$, and a proper smooth variety $X$ over $\mathbb{Q}$ with proper model $\mathfrak{X}$ over $\mathbb{Z}_{(p)}$, one has
\[
\log \zeta_p^{\mathrm{ss}}(X,s) = \sum_{r\geq 1} \sum_{x\in \mathcal{M}_m(\mathbb{F}_{p^r})} \mathrm{tr}^{\mathrm{ss}}(\Phi_{p^r}|(R\psi\overline{\mathbb{Q}}_{\ell})_x) \frac{p^{-rs}}{r}\ .
\]
Here $\Phi_{p^r}$ is a geometric Frobenius and $R\psi\overline{\mathbb{Q}}_{\ell}$ denotes the complex of nearby cycle sheaves. In the case that $\mathfrak{X}$ is smooth over $\mathbb{Z}_{(p)}$, this gives back \eqref{eqngoodred} since then $\zeta_p^{\mathrm{ss}}=\zeta_p$ and
\[
\mathrm{tr}^{\mathrm{ss}}(\Phi_{p^r}|(R\psi\overline{\mathbb{Q}}_{\ell})_x) = 1\ .
\]

Using the compatibility of the nearby cycles functor $R\psi$ with proper maps, we get in our situation
\[\log \zeta_p^{\mathrm{ss}}(\mathcal{M}_{\Gamma(p^n),m},s) = \sum_{r\geq 1} \sum_{x\in \mathcal{M}_m(\mathbb{F}_{p^r})} \mathrm{tr}^{\mathrm{ss}}(\Phi_{p^r}|(R\psi\mathcal{F}_n)_x) \frac{p^{-rs}}{r}\ ,\]
where $\mathcal{F}_n = \pi_{n\eta\ast} \overline{\mathbb{Q}}_{\ell}$ \footnote{For problems related to noncompactness of $\mathcal{M}_m$, see Theorem \ref{CohoVanish}.}. This essentially reduces the problem to that of computing the semisimple trace of Frobenius on the nearby cycle sheaves.

Our first result is a computation of the semisimple trace of Frobenius on the nearby cycles for certain regular schemes. Let $\mathcal{O}$ be the ring of integers in a local field $K$. Let $X/\mathcal{O}$ be regular and flat of relative dimension 1, with special fibre $X_s$. Let $X_{\eta^{\mathrm{ur}}}$ be the base-change of $X$ to the maximal unramified extension $K^{\mathrm{ur}}$ of $K$, let $X_{\mathcal{O}^{\mathrm{ur}}}$ be the base-change to the ring of integers in $K^{\mathrm{ur}}$ and let $X_{\overline{s}}$ be the geometric special fiber. Then we have $\iota: X_{\overline{s}}\longrightarrow X_{\mathcal{O}^{\mathrm{ur}}}$ and $j: X_{\eta^{\mathrm{ur}}}\longrightarrow X_{\mathcal{O}^{\mathrm{ur}}}$.

\begin{thmA} Assume that $X_s$ is globally the union of regular divisors. Let $x\in X_s(\mathbb{F}_q)$ and let $D_1,...,D_i$ be the divisors passing through $x$. Let $W_1$ be the $i$-dimensional $\overline{\mathbb{Q}}_{\ell}$-vector space with basis given by the $D_t$ and let $W_2$ be the kernel of the map $W_1\longrightarrow\overline{\mathbb{Q}}_{\ell}$ sending all $D_t$ to 1. Then there are canonical isomorphisms
\[(\iota^{\ast}R^kj_{\ast} \overline{\mathbb{Q}}_{\ell})_x\cong\left\{\begin{array}{ll} \overline{\mathbb{Q}}_{\ell} & k=0\\
W_1(-1) & k = 1\\
W_2(-2) & k = 2\\
0 & \mathrm{else}\ .\end{array} \right .\]
\end{thmA}

The main ingredient in the proof of this lemma is Thomason's purity theorem, \cite{Thomason}, a special case of Grothendieck's purity conjecture. Together with some general remarks made in Section \ref{RemarksSemisimple} this is enough to compute the semisimple trace of Frobenius. It is known that the assumptions of this lemma are fulfilled in the case of interest to us, as recalled in Section \ref{ModuliSpace}.

To state our second main result, we introduce some notation. For any integer $n\geq 0$, we define a function $\phi_{p,n}: \GL_2(\mathbb{Q}_{p^r})\longrightarrow \mathbb{Q}$. If $n=0$, it is simply $\frac{1}{p^r-1}$ times the characteristic function of the set
\[
\GL_2(\mathbb{Z}_{p^r})\left(\begin{array}{cc} p & 0 \\ 0 & 1 \end{array}\right)\GL_2(\mathbb{Z}_{p^r})\ .
\]
If $n>0$, we need further notation to state the definition. For $g\in \GL_2(\mathbb{Q}_{p^r})$, let $k(g)$ be the minimal integer $k$ such that $p^kg$ has integral entries. Further, let $\ell(g)=v_p(1-\tr g + \det g)$.\footnote{For a more conceptual interpretation of these numbers, see Section 14.} Then
\begin{itemize}
\item $\phi_{p,n}(g)=0$ except if $v_p(\det g)=1$, $v_p(\tr g)\geq 0$ and $k(g)\leq n-1$. Assume now that $g$ has these properties.
\item $\phi_{p,n}(g)=-1-p^r$ if $v_p(\tr g)\geq 1$,
\item $\phi_{p,n}(g)=1-p^{2\ell(g)r}$ if $v_p(\tr g)=0$ and $\ell(g)<n-k(g)$,
\item $\phi_{p,n}(g)=1+p^{(2(n-k(g))-1)r}$ if $v_p(\tr g)=0$ and $\ell(g)\geq n-k(g)$.
\end{itemize}

To any point $x\in \mathcal{M}_m(\mathbb{F}_{p^r})$, there is an associated element $\delta=\delta(x)\in \GL_2(\mathbb{Q}_{p^r})$, well-defined up to $\sigma$-conjugation. Its construction is based on crystalline cohomology and is recalled in Section 5.

Finally, let
\[
\Gamma(p^n)_{\mathbb{Q}_{p^r}} = \ker(\GL_n(\mathbb{Z}_{p^r})\longrightarrow \GL_n(\mathbb{Z}_{p^r}/p^n\mathbb{Z}_{p^r}) )\ .
\]
We normalize the Haar measure on $\GL_2(\mathbb{Q}_{p^r})$ by giving $\GL_2(\mathbb{Z}_{p^r})$ volume $p^r-1$.

\begin{thmB} \emph{(i)} The function $\phi_{p,n}$ lies in the center of the Hecke algebra of compactly supported functions on $\GL_2(\mathbb{Q}_{p^r})$ that are biinvariant under $\Gamma(p^n)_{\mathbb{Q}_{p^r}}$.

\emph{(ii)} For any point $x\in \mathcal{M}_m(\mathbb{F}_{p^r})$ with associated $\delta=\delta(x)$,
\[
\mathrm{tr}^{\mathrm{ss}}(\Phi_{p^r} | (R\psi\mathcal{F}_n)_x) = TO_{\delta\sigma}(\phi_{p,n})(TO_{\delta\sigma}(\phi_{p,0}))^{-1}\ .
\]

\emph{(iii)} For any irreducible admissible smooth representation $\pi$ of $\GL_2(\mathbb{Q}_{p^r})$ with
\[
\pi^{\Gamma(p^n)_{\mathbb{Q}_{p^r}}}\neq 0\ ,
\]the function $\phi_{p,n}$ acts through the scalar
\[
p^{\frac 12 r} \mathrm{tr}^{\mathrm{ss}}(\Phi_{p^r} | \sigma_{\pi})
\]
on $\pi^{\Gamma(p^n)_{\mathbb{Q}_{p^r}}}$, where $\sigma_{\pi}$ is the representation of the Weil-Deligne group of $\mathbb{Q}_{p^r}$ associated to $\pi$ by the Local Langlands Correspondence.
\end{thmB}

Part (ii) furnishes an explicit formula for the semisimple trace of Frobenius on the nearby cycles. Also note that the description of the Bernstein center implies that part (i) and (iii) uniquely characterize the function $\phi_{p,n}$. In fact, we will use this as the definition and then verify that it agrees with the explicit function only at the end, in Section 14.

This proves a conjecture of Haines and Kottwitz in the special case at hand. The conjecture states roughly that the semisimple trace of a power of Frobenius on the $\ell$-adic cohomology of a Shimura variety can be written as a sum of products of a volume factor, an orbital integral away from $p$ and a twisted orbital integral of a function in the center of a certain Hecke algebra. This is provided by Corollary \ref{LKNew} in our case, upon summing over all isogeny classes.

In order to proceed further, one has to relate the twisted orbital integrals to usual orbital integrals. To accomplish this, we prove a base-change identity for central functions.

Let
\[
\Gamma(p^n)_{\mathbb{Q}_p} = \ker(\GL_n(\mathbb{Z}_p)\longrightarrow \GL_n(\mathbb{Z}/p^n\mathbb{Z}) )\ .
\]
Further, for $G=\GL_n(\mathbb{Q}_p)$ or $G=\GL_n(\mathbb{Q}_{p^r})$, let $\mathcal{Z}(G)$ be the Bernstein center of $G$, see Section 2. For any compact open subgroup $K$, we denote by $e_K$ its associated idempotent.

\begin{thmC} Assume
\[
f\in \mathcal{Z}(\mathrm{GL}_2(\mathbb{Q}_p))\ ,\ \phi\in \mathcal{Z}(\mathrm{GL}_2(\mathbb{Q}_{p^r}))
\]
are given such that for every tempered irreducible smooth representation $\pi$ of $\mathrm{GL}_2(\mathbb{Q}_p)$ with base-change lift $\Pi$, the scalars $c_{f,\pi}$ resp. $c_{\phi,\Pi}$ through which $f$ resp. $\phi$ act on $\pi$ resp. $\Pi$, agree: $c_{f,\pi} = c_{\phi,\Pi}$.

Assume that $h\in C_c^{\infty}(\mathrm{GL}_2(\mathbb{Q}_p))$ and $h^{\prime}\in C_c^{\infty}(\mathrm{GL}_2(\mathbb{Q}_{p^r}))$ are such that the twisted orbital integrals of $h^{\prime}$ match with the orbital integrals of $h$, cf. Definition \ref{DefAssoc}. Then also $f\ast h$ and $\phi\ast h^{\prime}$ have matching (twisted) orbital integrals.

Furthermore, $e_{\Gamma(p^n)_{\mathbb{Q}_p}}$ and $e_{\Gamma(p^n)_{\mathbb{Q}_{p^r}}}$ have matching (twisted) orbital integrals.
\end{thmC}

This generalizes the corresponding result for a hyperspecial maximal compact subgroup, known as the base-change fundamental lemma. Versions of this result for general groups and parahoric subgroups have recently been obtained by Haines, \cite{HainesBC}.

Together with the Arthur-Selberg Trace Formula and an analysis of the contribution of the `points at infinity', Theorem B and Theorem C imply the following theorem. Recall that there is a smooth projective curve $\overline{\mathcal{M}}_m$ over $\mathbb{Z}[\frac 1m]$ containing $\mathcal{M}_m$ as a fiberwise open dense subset.

\begin{thmD} Assume that $m$ is the product of two coprime integers, both at least $3$. Then the Hasse-Weil zeta-function of $\overline{\mathcal{M}}_{m}$ is given by
\[
\zeta(\overline{\mathcal{M}}_{m},s) = \prod_{\pi\in \Pi_{\mathrm{disc}}(\mathrm{GL}_2(\mathbb{A}),1)} L(\pi,s-\tfrac 12)^{\frac 12 m(\pi) \chi(\pi_{\infty}) \dim \pi_f^{K_m}}\ ,
\]
where $\Pi_{\mathrm{disc}}(\mathrm{GL}_2(\mathbb{A}),1)$ is the set of automorphic representations
\[
\pi=\pi_f \otimes \pi_{\infty}
\]
of $\mathrm{GL}_2(\mathbb{A})$ that occur discretely in $L^2(\mathrm{GL}_2(\mathbb{Q})\mathbb{R}^{\times} \backslash \mathrm{GL}_2(\mathbb{A}))$ such that $\pi_{\infty}$ has trivial central and infinitesimal character. Furthermore, $m(\pi)$ is the multiplicity of $\pi$ inside $L^2(\mathrm{GL}_2(\mathbb{Q})\mathbb{R}^{\times} \backslash \mathrm{GL}_2(\mathbb{A}))$, $\chi(\pi_{\infty})=2$ if $\pi_{\infty}$ is a character and $\chi(\pi_{\infty})=-2$ otherwise, and
\[
K_m = \{g\in \mathrm{GL}_2(\hat{\mathbb{Z}}) \mid g\equiv 1\modd m\}\ .
\]
\end{thmD}

\begin{rem} Of course, multiplicity 1 for $\mathrm{GL}_2$ tells us that $m(\pi)=1$.
\end{rem}

A much stronger version of this theorem has been proved by Carayol in \cite{Carayol}. Decompose (the cuspidal part of) the $\ell$-adic cohomology of the modular curves according to automorphic representations $\pi=\otimes \pi_p$ as
\[
\bigoplus \pi\otimes \sigma_{\pi}\ ,
\]
where $\sigma_{\pi}$ is a $2$-dimensional representation of the absolute Galois group of $\mathbb{Q}$. Then Carayol determines the restriction of $\sigma_{\pi}$ to the absolute Galois group of $\mathbb{Q}_p$, $p\neq \ell$, by showing that it is paired (up to an explicit twist) with $\pi_p$ through the Local Langlands Correspondence. In particular, their $L$-functions agree up to shift, which gives our Theorem D upon taking the product over all automorphic representations $\pi$.

It would not be a serious problem to extend the methods used here to prove that all local $L$-factors of $\sigma_{\pi}$ and $\pi$ agree (up to shift), by allowing the action of arbitrary Hecke operators prime to $p$ in our considerations in order to `cut out' a single representation $\pi$ in the cohomology. If one could prove that the local $\epsilon$-factors of $\sigma_{\pi}$ and $\pi$ agree as well, this would give a new proof of Carayol's result, but we do not see any way to check this.

It should be pointed out, however, that Carayol uses advanced methods, relying on the `local fundamental representation' constructed by Deligne in \cite{Deligne}, strong statements about nearby cycles, the consideration of more general Shimura curves and some instances of automorphic functoriality, notably the Jacquet-Langlands correspondence and base-change for $\GL_2$.

By contrast, except for base-change for $\GL_2$, all of these methods are avoided in this article\footnote{In the form that our article is written, it makes use of (unramified) base-change for $\GL_2$, but this is needed only for Theorem B and Theorem C and could be avoided if one is only interested in Theorem D. Only the spherical base-change identity is really needed, whose proof reduces to explicit combinatorics as in \cite{LanglandsBCGL2}.}. Our approach relies on the geometry of the modular curve itself, the main geometric ingredient being Theorem A which relies on Thomason's purity theorem.

We now briefly describe the content of the different sections.

Section 2 up to Section 7 mainly recall results that will be needed later. Here, the first two sections are of a representation-theoretic nature, describing some results from local harmonic analysis, in particular the base-change identity, Theorem C. The next sections are of a more algebro-geometric nature, describing results about the moduli space of elliptic curves with level-structure and particularly their bad reduction, in Section 4 and 6. We also recall the Langlands-Kottwitz method of counting points in Section 5 and the definition of the semisimple local factor in Section 7.

The Sections 8 and 9 are technically the heart of this work. In Section 8, we prove our result on vanishing cycles, Theorem A, which allows us to compute the semisimple trace of Frobenius in the given situation. Then, in Section 9, we rewrite this result in terms of twisted orbital integrals of certain functions naturally defined through the local Langlands correspondence and prove Theorem B, modulo the explicit formula for $\phi_{p,n}$.

The rest of the article, Sections 10 to 13, employs the standard method of comparing the Lefschetz and the Arthur-Selberg Trace Formula to prove Theorem D.

Finally, Section 14 provides the explicit formula for the function $\phi_{p,n}$ and finishes the proof of Theorem B.

{\bf Notation}. For any field $K$, we denote by $G_K=\mathrm{Gal}(\overline{K}/K)$ its absolute Galois group.

{\bf Acknowledgments}. I wish to thank everyone who helped me writing this paper. In particular, I heartily thank my advisor M. Rapoport for introducing me to this topic and his constant encouragement and interest in these ideas.

\section{The Bernstein center}\label{BernsteinCenter}

Let $G=\mathrm{GL}_n(F)$, where $F$ is a local field. Let $\mathcal{H}(G,K)$ be the Hecke algebra of locally constant functions on $G$ with compact support and biinvariant under $K$, for a compact open subgroup $K$ of $G$.

We will recall the description of the center of the Hecke algebras $\mathcal{H}(G,K)$ where $K$ ranges through compact open subgroups of $G$, cf. \cite{Bernstein}. Denote the center by $\mathcal{Z}(G,K)$ and let $\mathcal{Z}(G)=\lim \limits_{\longleftarrow} \mathcal{Z}(G,K)$. Note that this is not a subset of the Hecke algebra $\mathcal{H}(G)$. Rather, it is a subset of $\widehat{\mathcal{H}}(G)=\lim \limits_{\longleftarrow} \mathcal{H}(G,K)\supset \mathcal{H}(G)$ which can be identified (after choosing a Haar measure) with the space of distributions $T$ of $G$ such that $T\ast e_K$ is of compact support for all compact open subgroups $K$. Here $e_K$ is the idempotent associated to $K$, i.e. the characteristic function of $K$ divided by its volume. Then $\widehat{\mathcal{H}}(G)$ has an algebra structure through convolution, and its center is $\mathcal{Z}(G)$. In fact, $\mathcal{Z}(G)$ consists of the conjugation-invariant distributions in $\widehat{\mathcal{H}}(G)$.

Let $\hat{G}$ be the set of irreducible smooth representations of $G$ over $\mathbb{C}$. By Schur's lemma, we have a map $\phi: \mathcal{Z}(G)\longrightarrow \mathrm{Map}(\hat{G},\mathbb{C}^{\times})$. We will now explain how to describe the center explicitly using this map.

Let $P$ be a parabolic subgroup of $G$ with Levi subgroup $L\cong \prod_{i=1}^k \mathrm{GL}_{n_i}$ and fix a supercuspidal representation $\sigma$ of $L$. Let $D=(\mathbb{G}_m)^k$. Then we have a universal unramified character $\chi: L\longrightarrow \Gamma(D,\mathcal{O}_D)\cong \mathbb{C}[T_1^{\pm 1},\ldots,T_k^{\pm 1}]$ sending $(g_i)_{i=1,\ldots,k}$ to $\prod_{i=1}^k T_i^{v_p(\det(g_i))}$. We get a corresponding family of representations $\text{n-Ind}_{P}^{G} (\sigma\chi)$ of $G$ parametrized by the scheme $D$ (here $\text{n-Ind}$ denotes the normalized induction). We will also write $D$ for the set of representations of $G$ one gets by specializing to a closed point of $D$.

Let $\mathrm{Rep}\ G$ be the category of smooth admissible representations of $G$ and let
\[
(\mathrm{Rep}\ G)(L,D)
\]
be the full subcategory of $\mathrm{Rep}\ G$ consisting of those representations that can be embedded into a direct sum of representations in the family $D$.

\begin{thm} $\mathrm{Rep}\ G$ is the direct sum of the categories $(\mathrm{Rep}\ G)(L,D)$ where $(L,D)$ are taken up to conjugation.
\end{thm}

\begin{proof} This is Proposition 2.10 in \cite{Bernstein}.
\end{proof}

Let $W(L,D)$ be the subgroup of $\mathrm{Norm}_G(L)/L$ consisting of those $n$ such that the set of representations $D$ coincides with its conjugate via $n$.

\begin{thm}\label{BernsteinCenterIsom} Fix a supercuspidal representation $\sigma$ of a Levi subgroup $L$ as above. Let $z\in \mathcal{Z}(G)$. Then $z$ acts by a scalar on $\emph{n-Ind}_{P}^{G}(\sigma\chi_0)$ for any character $\chi_0$. The corresponding function on $D$ is a $W(L,D)$-invariant regular function. This induces an isomorphism of $\mathcal{Z}(G)$ with the algebra of regular functions on $\bigcup \limits_{(L,D)} D/W(L,D)$.
\end{thm}

\begin{proof} This is Theorem 2.13 in \cite{Bernstein}.
\end{proof}

\section{Base change}

We will establish a base change identity that will be used later. This also allows us to recall certain facts about base change of representations.

Let $\sigma$ be the lift of Frobenius on $\mathbb{Q}_{p^r}$.

\begin{definition} For an element $\delta\in \mathrm{GL}_2(\mathbb{Q}_{p^r})$, we let $N\delta = \delta\delta^{\sigma}\cdots\delta^{\sigma^{r-1}}$.
\end{definition}

One easily sees that the conjugacy class of $N\delta$ contains an element of $\mathrm{GL}_2(\mathbb{Q}_p)$.

For $\gamma\in \mathrm{GL}_2(\mathbb{Q}_p)$, define the centralizer
\[ G_{\gamma}(R) = \{g\in \mathrm{GL}_2(R) \mid g^{-1}\gamma g = \gamma \} \]
and for $\delta\in \mathrm{GL}_2(\mathbb{Q}_{p^r})$ the twisted centralizer
\[ G_{\delta\sigma}(R) = \{h \in \mathrm{GL}_2(R\otimes \mathbb{Q}_{p^r}) \mid h^{-1} \delta h^{\sigma} = \delta \}\ . \]
It is known that $G_{\delta\sigma}$ is an inner form of $G_{N\delta}$. We choose associated Haar measures on their groups of $\mathbb{Q}_p$-valued points.

For any smooth function $f$ with compact support on $\mathrm{GL}_2(\mathbb{Q}_p)$, put
\[O_{\gamma}(f) = \int_{G_\gamma(\mathbb{Q}_p)\backslash \mathrm{GL}_2(\mathbb{Q}_p)} f(g^{-1}\gamma g) dg\]
and for any smooth function $\phi$ with compact support on $\mathrm{GL}_2(\mathbb{Q}_{p^r})$, put
\[TO_{\delta\sigma}(\phi) = \int_{G_{\delta\sigma}(\mathbb{Q}_p)\backslash \mathrm{GL}_2(\mathbb{Q}_{p^r})} \phi(h^{-1}\delta h^{\sigma}) dh\ .\]

\begin{definition}\label{DefAssoc} We say that functions
\[ \phi\in C_c^{\infty}(\mathrm{GL}_2(\mathbb{Q}_{p^r}))\ ,\ f\in C_c^{\infty}(\mathrm{GL}_2(\mathbb{Q}_p)) \]
have matching (twisted) orbital integrals (sometimes we simply say that they are `associated') if
\[ O_{\gamma}(f) = \left\{\begin{array}{ll} \pm TO_{\delta\sigma}(\phi)& \text{if}\ \gamma\ \text{is conjugate to}\ N\delta\ \text{for some}\ \delta\\ 0 & \text{else}\ , \end{array}\right . \]
for all semisimple $\gamma\in \mathrm{GL}_2(\mathbb{Q}_p)$. Here, the sign is $+$ except if $N\delta$ is a central element, but $\delta$ is not $\sigma$-conjugate to a central element, when it is $-$.
\end{definition}

\begin{rem} This definition depends on the choice of Haar measures on $\mathrm{GL}_2(\mathbb{Q}_p)$ and $\mathrm{GL}_2(\mathbb{Q}_{p^r})$ that we will not yet fix; it does not depend on the choice of Haar measures on $G_{\delta\sigma}(\mathbb{Q}_p)$ and $G_{N\delta}(\mathbb{Q}_p)$ as long as they are chosen compatibly.
\end{rem}

\begin{prop}\label{SizeTwistedvsNormal} Let $\delta\in \mathrm{GL}_2(\mathbb{Z}_{p^r}/p^n\mathbb{Z}_{p^r})$. Then
\[ G_{\delta\sigma}(\mathbb{Z}/p^n\mathbb{Z})=\{h\in \mathrm{GL}_2(\mathbb{Z}_{p^r}/p^n\mathbb{Z}_{p^r}) \mid h^{-1}\delta h^{\sigma} = \delta\} \]
has as many elements as
\[ G_{N\delta}(\mathbb{Z}/p^n\mathbb{Z})=\{g\in \mathrm{GL}_2(\mathbb{Z}/p^n\mathbb{Z}) \mid g^{-1}N\delta g=N\delta\}\ .\]
Furthermore, $\sigma$-conjugacy classes in $\mathrm{GL}_2(\mathbb{Z}_{p^r}/p^n\mathbb{Z}_{p^r})$ are mapped bijectively to conjugacy classes in $\mathrm{GL}_2(\mathbb{Z}/p^n\mathbb{Z})$ via the norm map.
\end{prop}

\begin{proof} Let $\gamma\in \mathrm{GL}_2(\mathbb{Z}/p^n\mathbb{Z})$. We get the commutative groups $Z_{\gamma,p}=(\mathbb{Z}/p^n\mathbb{Z}[\gamma])^{\times}$ and $Z_{\gamma,p^r}=(\mathbb{Z}_{p^r}/p^n\mathbb{Z}_{p^r}[\gamma])^{\times}$. The norm map defines a homomorphism $d_2: Z_{\gamma,p^r}\longrightarrow Z_{\gamma,p}$. Also define the homomorphism $d_1: Z_{\gamma,p^r}\longrightarrow Z_{\gamma,p^r}$ by $d_1(x) = xx^{-\sigma}$. By definition, we have
\[ H^1(\mathrm{Gal}(\mathbb{Q}_{p^r}/\mathbb{Q}_p),Z_{\gamma,p^r}) = \mathrm{ker}(d_2)/\mathrm{im}(d_1)\ .\]

\begin{lem} This cohomology group vanishes:
\[ H^1(\mathrm{Gal}(\mathbb{Q}_{p^r}/\mathbb{Q}_p),Z_{\gamma,p^r}) = 0\ .\]
Hence the following complex is exact
\[ 0\longrightarrow Z_{\gamma,p}\longrightarrow Z_{\gamma,p^r}\buildrel d_1\over \longrightarrow Z_{\gamma,p^r}\buildrel d_2\over \longrightarrow Z_{\gamma,p}\longrightarrow 0\ .\]
\end{lem}

\begin{proof} We have a $\mathrm{Gal}(\mathbb{Q}_{p^r}/\mathbb{Q}_p)$-invariant filtration on $Z_{\gamma,p^r}$ given by $X_i = \mathrm{ker}(Z_{\gamma,p^r}\longrightarrow \mathrm{GL}_2(\mathbb{Z}_{p^r}/p^i\mathbb{Z}_{p^r}))$ for $i=0,\ldots,n$. By the long exact sequence for cohomology, it is enough to prove the vanishing of the cohomology for the successive quotients. But for $i\geq 1$, the quotient $X_i/X_{i+1}$ is a $\mathbb{F}_{p^r}$-subvectorspace of
\[ \ker(\mathrm{GL}_2(\mathbb{Z}_{p^r}/p^{i+1}\mathbb{Z}_{p^r})\longrightarrow \mathrm{GL}_2(\mathbb{Z}_{p^r}/p^i\mathbb{Z}_{p^r}))\cong \mathbb{F}_{p^r}^4\ . \]
But by Lang's lemma,
\[ H^1(\mathrm{Gal}(\mathbb{Q}_{p^r}/\mathbb{Q}_p),\mathbb{F}_{p^r}) = 0\ .\]
For $i=0$, Lang's lemma works just as well, noting that the groups considered are connected.

The complex is clearly exact at the first two steps. We have just proved that it is exact at the third step. Hence the surjectivity of the last map follows by counting elements.
\end{proof}

Given $\gamma\in \mathrm{GL}_2(\mathbb{Z}/p^n\mathbb{Z})$, choose some $\delta\in Z_{\gamma,p^r}$ with $N\delta = \gamma$. This exists by the last lemma. We claim that in this case, $G_{\delta\sigma}(\mathbb{Z}/p^n \mathbb{Z}) = G_{\gamma}(\mathbb{Z}/p^n\mathbb{Z})$ as sets.

Take $x\in G_{\delta\sigma}(\mathbb{Z}/p^n\mathbb{Z})$. Then $x^{-1}\delta x^{\sigma} = \delta$ and hence $x^{-\sigma^i}\delta^{\sigma^i} x^{\sigma^{i+1}} = \delta^{\sigma^i}$ for all $i=0,\ldots,r-1$ and multiplying these equations gives
\[ x^{-1}N\delta x = N\delta \ ,\]
hence $x$ commutes with $\gamma=N\delta$. But then $x$ commutes with $\delta\in \mathbb{Z}_{p^r}/p^n\mathbb{Z}_{p^r}[\gamma]$ and therefore $x^{-1}\delta x^{\sigma} = \delta$ implies $x = x^{\sigma}$, whence $x\in G_{\gamma}(\mathbb{Z}/p^n\mathbb{Z})$.

The other inclusion $G_{\gamma}(\mathbb{Z}/p^n\mathbb{Z})\subset G_{\delta\sigma}(\mathbb{Z}/p^n\mathbb{Z})$ follows directly from $\delta\in \mathbb{Z}_{p^r}/p^n\mathbb{Z}_{p^r}[\gamma]$. This proves the claim $G_{\delta\sigma}(\mathbb{Z}/p^n \mathbb{Z}) = G_{\gamma}(\mathbb{Z}/p^n\mathbb{Z})$ and hence the first part of the Proposition in this case.

Now, for representatives $\gamma_1, \ldots, \gamma_t$ of the conjugacy classes in $\mathrm{GL}_2(\mathbb{Z}/p^n\mathbb{Z})$, we have constructed elements $\delta_1, \ldots, \delta_t$ with $N\delta_i = \gamma_i$ for all $i$, whence representing different $\sigma$-conjugacy classes. We know that the size of their $\sigma$-conjugacy classes is
\[ \frac{|\mathrm{GL}_2(\mathbb{Z}_{p^r}/p^n\mathbb{Z}_{p^r})|}{|G_{\delta_i\sigma}(\mathbb{Z}/p^n\mathbb{Z})|} = \frac{|\mathrm{GL}_2(\mathbb{Z}_{p^r}/p^n\mathbb{Z}_{p^r})|}{|G_{\gamma_i}(\mathbb{Z}/p^n\mathbb{Z})|}\ .\]
The sum gives
\[\begin{aligned} \frac{|\mathrm{GL}_2(\mathbb{Z}_{p^r}/p^n\mathbb{Z}_{p^r})|}{|\mathrm{GL}_2(\mathbb{Z}/p^n\mathbb{Z})|}\sum_{i=1}^t \frac {|\mathrm{GL}_2(\mathbb{Z}/p^n\mathbb{Z})|}{|G_{\gamma_i}(\mathbb{Z}/p^n\mathbb{Z})|} &= \frac{|\mathrm{GL}_2(\mathbb{Z}_{p^r}/p^n\mathbb{Z}_{p^r})|}{|\mathrm{GL}_2(\mathbb{Z}/p^n\mathbb{Z})|}|\mathrm{GL}_2(\mathbb{Z}/p^n\mathbb{Z})| \\
&= |\mathrm{GL}_2(\mathbb{Z}_{p^r}/p^n\mathbb{Z}_{p^r})| \ .
\end{aligned}\]
Hence every element of $\mathrm{GL}_2(\mathbb{Z}_{p^r}/p^n\mathbb{Z}_{p^r})$ is $\sigma$-conjugate to one of $\delta_1,\ldots, \delta_t$, proving the rest of the Proposition.
\end{proof}

We use this Proposition to prove the following identity. Define the principal congruence subgroups
\[
\Gamma(p^n)_{\mathbb{Q}_p} = \{g\in \mathrm{GL}_2(\mathbb{Z}_p)\mid g\equiv 1\modd p^k\}\ ,
\]
\[
\Gamma(p^n)_{\mathbb{Q}_{p^r}} = \{g\in \mathrm{GL}_2(\mathbb{Z}_{p^r})\mid g\equiv 1\modd p^k\}\ .
\]
For any compact open subgroup $K$ of $\mathrm{GL}_2(\mathbb{Q}_p)$ or $\mathrm{GL}_2(\mathbb{Q}_{p^r})$, let $e_K$ be the idempotent which is the characteristic function of $K$ divided by its volume.

\begin{cor}\label{BCUnit} Let $f$ be a conjugation-invariant locally integrable function on $\mathrm{GL}_2(\mathbb{Z}_p)$. Then the function $\phi$ on $\mathrm{GL}_2(\mathbb{Z}_{p^r})$ defined by $\phi(\delta)=f(N\delta)$ is locally integrable. Furthermore,
\[(e_{\Gamma(p^k)_{\mathbb{Q}_{p^r}}}\ast \phi)(\delta) = (e_{\Gamma(p^k)_{\mathbb{Q}_p}}\ast f)(N\delta)\]
for all $\delta\in \mathrm{GL}_2(\mathbb{Z}_{p^r})$.
\end{cor}

\begin{proof} Assume first that $f$ is locally constant, say invariant by $\Gamma(p^n)_{\mathbb{Q}_p}$. Of course, $\phi$ is then invariant under $\Gamma(p^n)_{\mathbb{Q}_{p^r}}$ and in particular locally integrable. The desired identity follows on combining Proposition \ref{SizeTwistedvsNormal} for the integers $k$ and $n$.

The corollary now follows by approximating $f$ by locally constant functions.
\end{proof}

Now we explain how to derive a base change fundamental lemma for elements in the center of Hecke algebras, once base change of representations is established.

Let tempered representations $\pi$, resp. $\Pi$, of $\mathrm{GL}_2(\mathbb{Q}_p)$, resp. $\mathrm{GL}_2(\mathbb{Q}_{p^r})$, be given.

\begin{definition} In this situation, $\Pi$ is called a base-change lift of $\pi$ if $\Pi$ is invariant under $\mathrm{Gal}(\mathbb{Q}_{p^r}/\mathbb{Q}_p)$ and for some extension of $\Pi$ to a representation of $\mathrm{GL}_2(\mathbb{Q}_{p^r})\rtimes \mathrm{Gal}(\mathbb{Q}_{p^r}/\mathbb{Q}_p)$, the identity
\[ \tr( Ng | \pi ) = \tr( (g,\sigma) | \Pi ) \]
holds for all $g\in \mathrm{GL}_2(\mathbb{Q}_{p^r})$ such that the conjugacy class of $Ng$ is regular semi-simple.
\end{definition}

It is known that there exist unique base-change lifts, cf. \cite{LanglandsBCGL2}, or more generally \cite{ArthurClozel}.

\begin{thm}\label{BaseChangeIdentity} Assume
\[
f\in \mathcal{Z}(\mathrm{GL}_2(\mathbb{Q}_p))\ ,\ \phi\in \mathcal{Z}(\mathrm{GL}_2(\mathbb{Q}_{p^r}))
\]
are given such that for every tempered irreducible smooth representation $\pi$ of $\mathrm{GL}_2(\mathbb{Q}_p)$ with base-change lift $\Pi$, the scalars $c_{f,\pi}$ resp. $c_{\phi,\Pi}$ through which $f$ resp. $\phi$ act on $\pi$ resp. $\Pi$, agree: $c_{f,\pi} = c_{\phi,\Pi}$.

Then for any associated $h\in C_c^{\infty}(\mathrm{GL}_2(\mathbb{Q}_p))$ and $h^{\prime}\in C_c^{\infty}(\mathrm{GL}_2(\mathbb{Q}_{p^r}))$, also $f\ast h$ and $\phi\ast h^{\prime}$ have matching (twisted) orbital integrals.

Furthermore, $e_{\Gamma(p^n)_{\mathbb{Q}_p}}$ and $e_{\Gamma(p^n)_{\mathbb{Q}_{p^r}}}$ are associated.
\end{thm}

\begin{proof} Because $h$ and $h^{\prime}$ are associated, we have $\tr( h | \pi ) = \tr ( (h^{\prime},\sigma) | \Pi )$ if $\Pi$ is a base-change lift of $\pi$, as follows from the Weyl integration formula, cf. \cite{LanglandsBCGL2}, p.99, for the twisted version. We find
\[
\tr( f\ast h | \pi ) = c_{f,\pi} \tr( h | \pi ) = c_{\phi,\Pi} \tr( (h^{\prime},\sigma) | \Pi ) = \tr( (\phi\ast h^{\prime},\sigma) | \Pi )\ .\]
We may find a function $f^{\prime}\in \mathcal{H}(\mathrm{GL}_2(\mathbb{Q}_p))$ that has matching (twisted) orbital integrals with $\phi\ast h^{\prime}$, cf. \cite{LanglandsBCGL2}, Prop. 6.2. This implies that $\tr( (\phi\ast h^{\prime},\sigma) | \Pi ) = \tr( f^{\prime} | \pi )$. Hence $\tr( (f\ast h - f^{\prime}) | \pi ) = 0$ for all tempered irreducible smooth representations $\pi$ of $\mathrm{GL}_2(\mathbb{Q}_p)$. By Kazhdan's density theorem, \cite{Kazhdan}, Theorem 1, all regular semi-simple orbital integrals of $f\ast h-f^{\prime}$ vanish. Hence $f\ast h$ and $\phi\ast h^{\prime}$ have matching regular semi-simple (twisted) orbital integrals. By \cite{Clozel}, Prop. 7.2, all semi-simple (twisted) orbital integrals of $f\ast h$ and $\phi\ast h^{\prime}$ match.

To show the last statement, we first check that
\[
\tr ( e_{\Gamma(p^n)_{\mathbb{Q}_p}} | \pi ) = \tr ( (e_{\Gamma(p^n)_{\mathbb{Q}_{p^r}}},\sigma) | \Pi )\ .
\]
But this follows directly from Corollary \ref{BCUnit} with $f$ the character of $\pi$ restricted to $\mathrm{GL}_2(\mathbb{Z}_p)$, $k=n$ and $\delta=1$, because characters are locally integrable. Now the rest of the argument is precisely as above.
\end{proof}

\section{The moduli space of elliptic curves with level structure: Case of good reduction}

We will briefly recall some aspects of the theory of the moduli space of elliptic curves with level structure that we shall need. All of the material presented in this section is contained in \cite{DeligneRapoport}.

\begin{definition} A morphism $p: E\longrightarrow S$ of schemes with a section $e: S\longrightarrow E$ is said to be an elliptic curve over $S$ if $p$ is proper, flat, and all geometric fibers are elliptic curves (with zero section given by $e$).
\end{definition}

We simply say that $E/S$ is an elliptic curve, omitting the morphisms $p$ and $e$ in the notation. It is well-known that an elliptic curve is canonically a commutative group scheme over $S$, with $e$ as unit section.

One might try to represent the functor
\[\begin{aligned}
\mathfrak{M}: (\mathrm{Schemes})&\longrightarrow (\mathrm{Sets})\\
S&\longmapsto \{E/S\ \text{elliptic curve over}\ S\ \text{up to isomorphism} \}\ ,
\end{aligned}\]
but it is well-known that this is not representable by a scheme. We need the next definition:

\begin{definition} A level-$m$-structure on an elliptic curve $E/S$ is an isomorphism of group schemes over $S$
\[\alpha: (\mathbb{Z}/m\mathbb{Z})^2_S\longrightarrow E[m]\ , \]
where $E[m]$ is the preimage of (the closed subscheme) $e$ under multiplication by $m: E\longrightarrow E$.
\end{definition}

This is motivated by the fact that for $S=\mathrm{Spec}\ k$ the spectrum of an algebraically closed field $k$ of characteristic prime to $m$, one always has (noncanonically) $E[m]\cong (\mathbb{Z}/m\mathbb{Z})^2$. However, for algebraically closed fields $k$ whose characteristic divides $m$, there are no level-$m$-structures at all and it follows that if $(E/S,\alpha)$ is an elliptic curve with level-$m$-structure then $m$ is invertible on $S$. Consider now the following functor
\[\mathfrak{M}_m: (\mathrm{Schemes}/\mathbb{Z}[m^{-1}])\longrightarrow (\mathrm{Sets}) \]
\[ S\longmapsto \left\{\begin{array}{l}(E/S,\alpha)\ \text{elliptic curve}\ E\ \text{over}\ S\ \text{with}\\ \text{level-}m\text{-structure}\ \alpha\text{, up to isomorphism} \end{array} \right\}\ . \]

\begin{thm} For $m\geq 3$, the functor $\mathfrak{M}_m$ is representable by a smooth affine curve $\mathcal{M}_m$ over $\mathrm{Spec}\ \mathbb{Z}[\frac 1m]$. There is a projective smooth curve $\overline{\mathcal{M}}_m$ containing $\mathcal{M}_m$ as an open dense subset such that the boundary $\partial \mathcal{M}_m = \overline{\mathcal{M}}_m \setminus \mathcal{M}_m$ is \'{e}tale over $\mathrm{Spec}\ \mathbb{Z}[\frac 1m]$.
\end{thm}

Because the integer $m$ plays a minor role in the following, we will write $\mathcal{M}$ for $\mathcal{M}_m$.

\section{Counting points: The Langlands-Kottwitz approach}

We will explain the method of Langlands-Kottwitz to count the number of points $\modd p$ of Shimura varieties with good reduction, in the case of the modular curve. This is based on some unpublished notes of Kottwitz \cite{KottwitzNotes}.

Let $p$ be a prime not dividing $m$. Fix an elliptic curve $E_0$ over $\mathbb{F}_{p^r}$, for some positive integer $r$. Let $\mathbb{A}_f^p$ be the ring of finite ad\`eles of $\mathbb{Q}$ with trivial $p$-component and $\hat{\mathbb{Z}}^p\cong \prod_{\ell \neq p} \mathbb{Z}_{\ell}$ be the integral elements in $\mathbb{A}_f^p$.

We want to count the number of elements of
\[
\mathcal{M}(\mathbb{F}_{p^r})(E_0):=\{x\in \mathcal{M}(\mathbb{F}_{p^r}) \mid E_x\ \text{is}\ \mathbb{F}_{p^r}\text{-isogeneous to}\ E_0\}\ .
\]
Define
\[
H^p = H^1_{\mathrm{et}}(E_0,\mathbb{A}_f^p)\ ,\ H_p = H^1_{\mathrm{cris}}(E_0/\mathbb{Z}_{p^r})\otimes_{\mathbb{Z}_{p^r}} \mathbb{Q}_{p^r}\ .
\]
Now take $x\in \mathcal{M}(\mathbb{F}_{p^r})(E_0)$ arbitary. Choosing an $\mathbb{F}_{p^r}$-isogeny $f: E_0\longrightarrow E_x$, we get a $G_{\mathbb{F}_{p^r}}=\mathrm{Gal}(\overline{\mathbb{F}}_{p^r}/\mathbb{F}_{p^r})$-invariant $\hat{\mathbb{Z}}^p$-lattice
\[
L=f^{\ast}(H^1_{\mathrm{et}}(E_x,\hat{\mathbb{Z}}^p))\subset H^p\ ,
\]
an $F,V$-invariant $\mathbb{Z}_{p^r}$-lattice
\[
\Lambda=f^{\ast}(H^1_{\mathrm{cris}}(E_x/\mathbb{Z}_{p^r}))\subset H_p\ ,
\]
and a $G_{\mathbb{F}_{p^r}}$-invariant isomorphism
\[
\phi: (\mathbb{Z}/m\mathbb{Z})^2 \longrightarrow L\otimes \mathbb{Z}/m\mathbb{Z}
\]
(where the right hand side has the trivial $G_{\mathbb{F}_{p^r}}$-action), corresponding to the level-$m$-structure. Let $Y^p$ be the set of such $(L,\phi)$ and $Y_p$ be the set of $\Lambda$ as above. Dividing by the choice of $f$, we get a map
\[
\mathcal{M}(\mathbb{F}_{p^r})(E_0)\longrightarrow \Gamma\backslash Y^p \times Y_p\ ,
\]
where $\Gamma = (\mathrm{End}(E_0)\otimes \mathbb{Q})^{\times}$.

\begin{thm} This map is a bijection.
\end{thm}

\begin{proof} Assume that $(E_1,\phi_1)$ and $(E_2,\phi_2)$ have the same image. Choose isogenies $f_1: E_0\longrightarrow E_1$, $f_2: E_0\longrightarrow E_2$. Then the corresponding elements of $Y^p\times Y_p$ differ by an element $h\in \Gamma$. Write $h=m^{-1}h_0$ where $m$ is an integer and $h_0$ is a self-isogeny of $E_0$. Changing $f_1$ to $f_1h$ and $f_2$ to $f_2m$, we may assume that the elements of $Y^p\times Y_p$ are the same. We want to see that $f=f_1f_2^{-1}$, a priori an element of $\mathrm{Hom}(E_2,E_1)\otimes \mathbb{Q}$, actually belongs to $\mathrm{Hom}(E_2,E_1)$. Analogously, $f_2f_1^{-1}$ will be an actual morphism, so that they define inverse isomorphisms.

Now, let $M$ be an integer such that $Mf: E_2\longrightarrow E_1$ is an isogeny. Our knowledge of what happens on the cohomology implies by the theory of \'{e}tale covers of $E_1$ and the theory of Dieudonn\'{e} modules, that $Mf$ factors through multiplication by $M$. This is what we wanted to show. Note that $\phi_1$ and $\phi_2$ are carried to each other by assumption.

For surjectivity, let $(L,\phi,\Lambda)\in Y^p\times Y_p$ be given. By changing by a scalar, we may assume that $L$ and $\Lambda$ are contained in the integral lattices
\[H^1_{\mathrm{et}}(E_0,\hat{\mathbb{Z}}^p)\ ,\ H^1_{\mathrm{cris}}(E_0/\mathbb{Z}_{p^r})\ .\]
Then the theory of Dieudonn\'{e} modules provides us with a subgroup of $p$-power order $G_p$ corresponding to $\Lambda$ and the theory of \'{e}tale covers of $E_0$ provides us with a subgroup $G^p$ of order prime to $p$, corresponding to $L$. We then take $E_1 = E_0/G^pG_p$. It is easy to see that this gives the correct lattices. Of course, $\phi$ provides a level-$m$-structure.
\end{proof}

From here, it is straightforward to deduce the following corollary. Let $\gamma\in \mathrm{GL}_2(\mathbb{A}_f^p)$ be the endomorphism induced by $\Phi_{p^r}$ on $H^p$ (after choosing a basis of $H^p$). Similarly, let $\delta\in \mathrm{GL}_2(\mathbb{Q}_{p^r})$ be induced by the $p$-linear endomorphism $F$ on $H_p$ (after choosing a basis of $H_p$): If $\sigma$ is the $p$-linear isomorphism of $H_p$ preserving the chosen basis, define $\delta$ by $F=\delta\sigma$. Then we have the centralizer
\[ G_{\gamma}(\mathbb{A}_f^p) = \{g\in \mathrm{GL}_2(\mathbb{A}_f^p) \mid g^{-1}\gamma g = \gamma \} \]
of $\gamma$ in $\mathrm{GL}_2(\mathbb{A}_f^p)$ and the twisted centralizer
\[ G_{\delta\sigma}(\mathbb{Q}_p) = \{h \in \mathrm{GL}_2(\mathbb{Q}_{p^r}) \mid h^{-1} \delta h^{\sigma} = \delta \} \]
of $\delta$ in $\mathrm{GL}_2(\mathbb{Q}_{p^r})$. Let $f^p$ be the characteristic function of the set
\[K^p = \{g\in \mathrm{GL}_2(\hat{\mathbb{Z}}^p)\mid g\equiv 1\modd m\}\]
divided by its volume and let $\phi_{p,0}$ be the characteristic function of the set
\[\mathrm{GL}_2(\mathbb{Z}_{p^r})\left(\begin{array}{cc} p & 0 \\ 0 & 1 \end{array}\right) \mathrm{GL}_2(\mathbb{Z}_{p^r})\]
divided by the volume of $\mathrm{GL}_2(\mathbb{Z}_{p^r})$. For any smooth function with compact support $f$ on $\mathrm{GL}_2(\mathbb{A}_f^p)$, put
\[O_{\gamma}(f) = \int_{G_\gamma(\mathbb{A}_f^p)\backslash \mathrm{GL}_2(\mathbb{A}_f^p)} f(g^{-1}\gamma g) dg\ .\]

\begin{cor}\label{LKClassical} The cardinality of $\mathcal{M}(\mathbb{F}_{p^r})(E_0)$ is
\[\mathrm{vol}(\Gamma\backslash G_{\gamma}(\mathbb{A}_f^p)\times G_{\delta\sigma}(\mathbb{Q}_p)) O_{\gamma}(f^p) TO_{\delta\sigma}(\phi_{p,0})\ ,\]
where the Haar measure on $\Gamma$ gives points measure 1.
\end{cor}

\begin{proof} Choose the integral cohomology of $E_0$ as a base point in $Y^p$ and $Y_p$. Then we may identify the set $X^p$ of pairs $(L,\phi)$ as above, but without the Galois-invariance condition, with $\mathrm{GL}_2(\mathbb{A}_f^p)/K^p$. Similarly, we may identify $X_p$, the set of all lattices $\Lambda$, with $\mathrm{GL}_2(\mathbb{Q}_{p^r})/K_p$, where
\[ K_p = \mathrm{GL}_2(\mathbb{Z}_{p^r})\ .\]
The condition that an element $gK^p$ of $X^p$ lies in $Y^p$ is then expressed by saying that $\gamma gK^p = gK^p$, or equivalently $g^{-1}\gamma g\in K^p$. Similarly, the condition that an element $hK_p$ of $X_p$ lies in $Y_p$ is expressed by $FhK_p\subset hK_p$ and $VhK_p\subset hK_p$. Noting that $FV=p$, this is equivalent to $phK_p\subset FhK_p\subset hK_p$, i.e.
\[
pK_p\subset h^{-1}\delta h^{\sigma}K_p\subset K_p\ .
\]
The Weil pairing gives an isomorphism of the second exterior power of $H_p$ with $\mathbb{Q}_{p^r}(-1)$, so that $v_p(\det \delta)=1$. In particular, the condition on $h$ can be rewritten as
\[ h^{-1}\delta h^{\sigma}\in K_p \left(\begin{array}{cc} p & 0 \\ 0 & 1 \end{array}\right) K_p\ .\]
This means that the cardinality of $\Gamma\backslash Y^p\times Y_p$ is equal to
\[\int_{\Gamma\backslash \mathrm{GL}_2(\mathbb{A}_f^p)\times \mathrm{GL}_2(\mathbb{Q}_{p^r})} f^p(g^{-1}\gamma g)\phi_{p,0}(h^{-1}\delta h^{\sigma}) dg dh\ .\]
The formula of the corollary is a simple transcription.
\end{proof}

\begin{rem}\label{NonvanishingTO} In particular $TO_{\delta\sigma}(\phi_{p,0})\neq 0$ whenever $\mathcal{M}(\mathbb{F}_{p^r})(E_0)\neq \emptyset$.
\end{rem}

\section{The moduli space of elliptic curves with level structure: Case of bad reduction}\label{ModuliSpace}

We are interested in extending the moduli spaces $\mathcal{M}_m$, defined over $\mathrm{Spec}\ \mathbb{Z}[\frac 1m]$, to the remaining primes, where they have bad reduction. The material presented here is contained in \cite{KatzMazur}. Let us fix a prime $p$ first and choose some integer $m\geq 3$ prime to $p$. For any integer $n\geq 0$, we want to extend the scheme $\mathcal{M}_{p^n m}$ over $\mathrm{Spec}\ \mathbb{Z}[\frac 1m]$, noting that we already have defined it over $\mathrm{Spec}\ \mathbb{Z}[\frac 1{pm}]$.

\begin{definition} A Drinfeld-level-$p^n$-structure on an elliptic curve $E/S$ is a pair of sections $P,Q: S\longrightarrow E[p^n]$ such that there is an equality of relative Cartier divisors
\[ \sum_{i,j\in \mathbb{Z}/p^n\mathbb{Z}} [iP+jQ] = E[p^n]\ .\]
\end{definition}

Since for $p$ invertible on $S$, the group scheme $E[p^n]$ is \'{e}tale over $S$, a Drinfeld-level-$p^n$-structure coincides with an ordinary level-$p^n$-structure in this case. Hence the following gives an extension of the functor $\mathfrak{M}_{p^nm}$ to schemes over $\mathrm{Spec}\ \mathbb{Z}[\frac 1m]$:

\[ \mathfrak{M}_{\Gamma(p^n),m}: (\mathrm{Schemes}/ \mathbb{Z}[m^{-1}])\longrightarrow (\mathrm{Sets}) \]
\[ S\longmapsto \left\{\begin{array}{l}(E/S,(P,Q),\alpha)\ \text{elliptic curve}\ E\ \text{over}\ S\ \text{with}\\
\text{Drinfeld-level-}p^n\text{-structure}\ (P,Q)\ \text{and} \\
\text{level-}m\text{-structure}\ \alpha\text{, up to isomorphism} \end{array} \right\}\ .\]

\begin{thm} The functor $\mathfrak{M}_{\Gamma(p^n),m}$ is representable by a regular scheme $\mathcal{M}_{\Gamma(p^n),m}$ which is an affine curve over $\mathrm{Spec}\ \mathbb{Z}[\frac 1m]$. The canonical (forgetful) map
\[
\pi_n: \mathcal{M}_{\Gamma(p^n),m}\longrightarrow \mathcal{M}_m
\]
is finite. Over $\mathrm{Spec}\ \mathbb{Z}[\frac 1{pm}]$, it is an \'{e}tale cover with Galois group $\mathrm{GL}_2(\mathbb{Z}/p^n\mathbb{Z})$.
\end{thm}

Again, the integer $m$ plays a minor role, so we will suppress it from the notation and write $\mathcal{M}_{\Gamma(p^n)}$ for $\mathcal{M}_{\Gamma(p^n),m}$.

In this situation, the problem of compactification is slightly more difficult. Recall that the Weil pairing is a perfect pairing
\[ E[p^n]\times_S E[p^n]\longrightarrow \mu_{p^n,S}\ . \]
It allows us to define a morphism
\[ \mathcal{M}_{\Gamma(p^n)}\longrightarrow \mathrm{Spec}\ \mathbb{Z}[m^{-1}][\zeta_{p^n}]\ ,\]
where $\zeta_{p^n}$ is a primitive $p^n$-th root of unity, by sending $\zeta_{p^n}$ to the image of the universal sections $(P,Q)$ under the Weil pairing.

\begin{thm}\label{Compactification} There is a smooth proper curve $\overline{\mathcal{M}}_{\Gamma(p^n)}/ \mathbb{Z}[m^{-1}][\zeta_{p^n}]$ with $\mathcal{M}_{\Gamma(p^n)}$ as an open subset such that the complement is \'{e}tale over $\mathrm{Spec}\ \mathbb{Z}[m^{-1}][\zeta_{p^n}]$ and has a smooth neighborhood.
\end{thm}

We end this section with a description of the special fiber in characteristic $p$ of $\mathcal{M}_{\Gamma(p^n)}$. For any direct summand $H\subset (\mathbb{Z}/p^n\mathbb{Z})^2$ of order $p^n$, write $\mathcal{M}_{\Gamma(p^n)}^H$ for the reduced subscheme of the closed subscheme of $\mathcal{M}_{\Gamma(p^n)}$ where
\[\sum_{(i,j)\in H\subset (\mathbb{Z}/p^n\mathbb{Z})^2} [iP+jQ] = p^n[e]\ .\]

\begin{thm}\label{SpecialFiber} For any $H$, the closed subscheme $\mathcal{M}_{\Gamma(p^n)}^H$ is a regular divisor on $\mathcal{M}_{\Gamma(p^n)}$ which is supported in $\mathcal{M}_{\Gamma(p^n)}\otimes_{\mathbb{Z}} \mathbb{F}_p$. Any two of them intersect exactly at the supersingular points of $\mathcal{M}_{\Gamma(p^n)}\otimes_{\mathbb{Z}} \mathbb{F}_p$, i.e. those points such that the associated elliptic curve is supersingular. Furthermore,
\[ \mathcal{M}_{\Gamma(p^n)}\otimes_{\mathbb{Z}} \mathbb{F}_p = \bigcup_H \mathcal{M}_{\Gamma(p^n)}^H\ .\]
\end{thm}

\section{The (semisimple) local factor}\label{RemarksSemisimple}

In this section, we want to recall certain invariants attached to (the cohomology) of a variety $X$ over a local field $K$ with residue field $\mathbb{F}_q$. Recall that we denoted $G_K=\mathrm{Gal}(\overline{K}/K)$. Further, let $I_K\subset G_K$ be the inertia subgroup and let $\Phi_q$ be a geometric Frobenius element.

Let $\ell$ be a prime which does not divide $q$.

\begin{definition} The Hasse-Weil local factor of $X$ is
\[ \zeta(X,s)=\prod_{i=0}^{2\dim X} \det(1-\Phi_q q^{-s} | H_c^i(X\otimes_K \overline{K},\overline{\mathbb{Q}}_{\ell})^{I_K} )^{(-1)^{i+1}}\ .\]
\end{definition}

Here $H_c^i$ denotes \'{e}tale cohomology with compact supports.

Note that this definition depends on $\ell$; it is however conjectured that it is independent of $\ell$, as follows from the monodromy conjecture. As we are working only with curves and the monodromy conjecture for curves is proven in \cite{RapoportZink}, we get no problems.

It is rather hard to compute the local factors if $X$ has bad reduction. However, there is a slight variant which comes down to counting points `with multiplicity'. For this, we need to introduce the concept of semisimple trace, for which we also refer the reader to \cite{HainesNgo}.

Let $V$ be a continuous representation of $G_K$ in a finite dimensional $\overline{\mathbb{Q}}_{\ell}$-vector space, where $\ell$ is prime to the residue characteristic of $K$. Furthermore, let $H$ be a finite group acting on $V$, commuting with the action of $G_K$.

\begin{lem}\label{AdmFiltration} There is a filtration
\[ 0 = V_0\subset V_1\subset \cdots \subset V_k = V \]
into $G_K\times H$-invariant subspaces $V_i$ such that $I_K$ acts through a finite quotient on $\mathrm{gr}\ V_{\bullet} = \bigoplus_{i=1}^k V_i/V_{i-1}$.
\end{lem}

\begin{proof} Note that this contains Grothendieck's local monodromy theorem, \cite{SGA7I}. We will repeat the proof here. By induction, it suffices to find a nonzero $G_K\times H$-stable subspace $V_1$ on which $I_K$ acts through a finite quotient. In fact, it is enough to find a $I_K\times H$-stable subspace with this property, as the maximal $I_K\times H$-stable subspace on which $I_K$ acts through a finite quotient is automatically $G_K\times H$-stable, because $I_K\times H$ is normal in $G_K\times H$.

First, we check that the image of $I_K\times H$ is contained in $\mathrm{GL}_n(E)$ for some finite extension $E$ of $\mathbb{Q}_{\ell}$. Denote $\rho: I_K\times H\longrightarrow \mathrm{GL}_n(\overline{\mathbb{Q}}_{\ell})$. Since $I_K\times H$ is (locally) compact and hausdorff, it is a Baire space, i.e. the intersection of countably many dense open subsets is nonempty. Assume that there was no such extension $E$. For all $E$,
\[
\rho^{-1}(\mathrm{GL}_n(\overline{\mathbb{Q}}_{\ell}) \setminus \mathrm{GL}_n(E))
\]
is an open subset of $I_K\times H$. Clearly, their intersection is empty and there are only countably many finite extensions $E$ of $\mathbb{Q}_{\ell}$ inside $\overline{\mathbb{Q}}_{\ell}$. Hence one of them is not dense. But then some subgroup of finite index maps to $\mathrm{GL}_n(E)$, which easily implies the claim, after passing to a finite extension.

Since $I_K\times H$ is compact, the map $\rho: I_K\times H\longrightarrow \mathrm{GL}_n(E)$ factors through some maximal compact subgroup, which after conjugation may be assumed to be $\mathrm{GL}_n(\mathcal{O})$, where $\mathcal{O}$ is the ring of integral elements of $E$. Let $\mathbb{F}$ be the residue field of $E$.

There is a surjection $t: I_K\longrightarrow \mathbb{Z}_{\ell}$ whose kernel $I_K^{\ell}$ is an inverse limit of groups of order prime to $\ell$. But the kernel of the map $\mathrm{GL}_n(\mathcal{O})\longrightarrow \mathrm{GL}_n(\mathbb{F})$ is a pro-$\ell$-group and hence meets $I_K^{\ell}$ trivially. This means that $I_K^{\ell}$ acts through a finite quotient on $V$. Let $I_K^{\ell\prime}$ be the kernel of $I_K^{\ell}\longrightarrow \mathrm{GL}_n(\mathcal{O})$ and let $I_K^{\prime}=I_K/I_K^{\ell\prime}$ have center $Z$. Our considerations show that $t|_Z: Z\longrightarrow \mathbb{Z}_{\ell}$ is nontrivial and has finite kernel.

Let $\lambda\in Z$ with $t(\lambda)\neq 0$. Recall that $\Phi_q^{-1} t(g) \Phi_q = q t(g)$ for all $g\in I_K$. In particular, there are positive integers $r$ and $s$ such that $\Phi_q^{-s}\lambda^{q^r} \Phi_q^s=\lambda^{q^{r+s}}$, so that the image $\rho(\lambda)^{q^r}$ in $\mathrm{GL}_n(E)$ is conjugate to $\rho(\lambda)^{q^{r+s}}$. This implies that all eigenvalues of $\rho(\lambda)$ are roots of unity, so that by replacing $\lambda$ by a power, we may assume that $W=V^{\lambda=1}$ is nontrivial. But since $\lambda\in Z$, $W$ is $I_K\times H$-stable.
\end{proof}

\begin{definition} For $h\in H$, we define
\[ \mathrm{tr}^{\mathrm{ss}}(\Phi_q^r h | V) = \mathrm{tr}(\Phi_q^r h | (\mathrm{gr}\ V_{\bullet})^{I_K}) \]
for any filtration $V_{\bullet}$ as in the previous lemma.
\end{definition}

\begin{prop} This definition is independent of the choice of the filtration. In particular, the semisimple trace is additive in short exact sequences.
\end{prop}

\begin{proof} Taking a common refinement of two filtrations, this reduces to the well-known statement that for any endomorphism $\phi$ of a vector space $V$ with $\phi$-invariant subspace $W$, one has
\[ \tr (\phi | V) = \tr (\phi | W) + \tr (\phi| V/W)\ .\]
\end{proof}

This allows one to define the semisimple trace on the Grothendieck group, or on the derived category of finite-dimensional continuous $\ell$-adic representations of $G_K\times H$.

Next, we explain a different point of view on the semisimple trace. Let us consider the bounded derived category $D^b(\mathrm{Rep}_{\overline{\mathbb{Q}}_{\ell}} (G_K\times H))$ of continuous representations of $G_K\times H$ in finite dimensional $\overline{\mathbb{Q}}_{\ell}$-vector spaces.

\begin{rem}\label{DerivedStrange} Note that the correct version of the derived category of $\ell$-adic sheaves on a scheme $X$ is defined as direct $2$-limit over all finite extensions $E\subset\overline{\mathbb{Q}}_{\ell}$ of $\mathbb{Q}_{\ell}$ of the inverse $2$-limit of the derived categories of constructible $\mathbb{Z}/\ell^n\mathbb{Z}$-sheaves, tensored with $E$. We use the same definition of $D^b(\mathrm{Rep}_{\overline{\mathbb{Q}}_{\ell}} (G_K\times H))$ as the direct $2$-limit of the inverse $2$-limit of $D^b(\mathrm{Rep}_{\mathbb{Z}/\ell^n\mathbb{Z}} (G_K\times H))$, tensored with $E$, here. See \cite{KiehlWeissauer}, Chapter 2, for a detailed discussion.
\end{rem}

Consider the derived functor
\[
R_{I_K}: D^b(\mathrm{Rep}_{\overline{\mathbb{Q}}_{\ell}} (G_K\times H))\longrightarrow D^b(\mathrm{Rep}_{\overline{\mathbb{Q}}_{\ell}} (G_{\mathbb{F}_q} \times H))
\]
of taking invariants under $I_K$.

\begin{rem} Again, this is abuse of language as only with finite coefficients, this really is the derived functor.
\end{rem}

The finiteness properties needed here are special cases of the finiteness theorems for \'{e}tale cohomology: Consider
\[
\mathrm{Spec}\ \overline{\mathbb{F}}_q\buildrel \iota \over \longrightarrow \mathrm{Spec}\ \mathcal{O}^{\mathrm{ur}}\buildrel j\over \longleftarrow \mathrm{Spec}\ K^{\mathrm{ur}}\ ,
\]
where $K^{\mathrm{ur}}$ is the maximal unramified extension of $K$ and $\mathcal{O}^{\mathrm{ur}}$ are its integral elements. Then
\[
R_{I_K} = \iota^{\ast} Rj_{\ast}\ .
\]
We have defined a map
\[ \mathrm{tr}^{\mathrm{ss}}(\Phi_q^r h): D^b(\mathrm{Rep}_{\overline{\mathbb{Q}}_{\ell}} (G_K\times H))\longrightarrow \overline{\mathbb{Q}}_{\ell} \]
that is additive in distinguished triangles. There is a second map
\[ \mathrm{tr}(\Phi_q^r h)\circ R_{I_K}: D^b(\mathrm{Rep}_{\overline{\mathbb{Q}}_{\ell}} (G_K\times H))\longrightarrow \overline{\mathbb{Q}}_{\ell}\ . \]
Again, it is additive in distinguished triangles.

\begin{lem} These two linear forms are related by
\[ \mathrm{tr}(\Phi_q^r h)\circ R_{I_K} = (1-q^r) \mathrm{tr}^{\mathrm{ss}}(\Phi_q^r h)\ . \]
\end{lem}

\begin{proof} Because of the additivity of both sides and the existence of filtrations as in Lemma \ref{AdmFiltration}, it suffices to check this for a complex
\[\ldots \longrightarrow 0\longrightarrow V_0\longrightarrow 0\longrightarrow \ldots\]
concentrated in degree 0 and with $I_K$ acting through a finite quotient on $V_0$. We can even assume that this quotient is cyclic, as taking invariants under the wild inertia subgroup is exact, and the tame inertia group is procyclic. Then $\mathrm{tr}^{\mathrm{ss}}(\Phi_q^r h|V_0) = \mathrm{tr}(\Phi_q^r h|V_0^{I_K})$, and $R_{I_K}(V_0)$ is represented by the complex
\[\ldots \longrightarrow 0\longrightarrow V_0^{I_K}\buildrel {0}\over \longrightarrow V_0^{I_K}(-1)\longrightarrow 0\longrightarrow \ldots.\]
The lemma is now obvious.
\end{proof}

Let $X$ be a variety over $K$.

\begin{definition} The semisimple local factor is defined by
\[ \log \zeta^{\mathrm{ss}}(X,s) = \sum_{r\geq 1} \sum_{i=0}^{2\dim X} (-1)^i\mathrm{tr}^{\mathrm{ss}}(\Phi_q^r | H_c^i(X\otimes_K \overline{K},\overline{\mathbb{Q}}_{\ell})) \frac{q^{-rs}}{r}\ .\]
\end{definition}

Note that if $I_K$ acts through a finite quotient (e.g., if it acts trivially, as is the case when $X$ has good reduction), this agrees with the usual local factor.

Let $\mathcal{O}\subset K$ be the ring of integers. For a scheme $X_{\mathcal{O}}/\mathcal{O}$ of finite type, we write $X_s$, $X_{\overline{s}}$, $X_{\eta}$ resp. $X_{\overline{\eta}}$ for its special, geometric special, generic resp. geometric generic fiber. Let $X_{\overline{\mathcal{O}}}$ denote the base change to the ring of integers in the algebraic closure of $K$. Then we have maps $\overline{\iota}: X_{\overline{s}}\longrightarrow X_{\overline{\mathcal{O}}}$ and $\overline{j}: X_{\overline{\eta}}\longrightarrow X_{\overline{\mathcal{O}}}$.

\begin{definition} For a $\overline{\mathbb{Q}}_{\ell}$-sheaf $\mathcal{F}$ on $X_{\eta}$, the complex of nearby cycle sheaves is defined to be
\[ R\psi\mathcal{F} = \overline{\iota}^{\ast} R\overline{j}_{\ast} \mathcal{F}_{\overline{\eta}}\ ,\]
where $\mathcal{F}_{\overline{\eta}}$ is the pullback of $\mathcal{F}$ to $X_{\overline{\eta}}$. This is an element of the (so-called) derived category of $\overline{\mathbb{Q}}_{\ell}$-sheaves on $X_{\overline{s}}$ with an action of $G_K$ that is compatible with its action on $X_{\overline{s}}$.
\end{definition}

\begin{thm}\label{BadReduction} Assume that $X_{\mathcal{O}} / \mathcal{O}$ is a scheme of finite type such that there exists an open immersion $X_{\mathcal{O}}\subset \overline{X}_{\mathcal{O}}$ where $\overline{X}_{\mathcal{O}}$ is proper over $\mathcal{O}$, with complement $D$ a relative normal crossings divisor (i.e. there is an open neighborhood $U$ of $D$ in $\overline{X}_{\mathcal{O}}$ which is smooth over $\mathcal{O}$, such that $D$ is a relative normal crossings divisor in $U$). Then there is a canonical $G_K$-equivariant isomorphism
\[H_c^i(X_{\overline{\eta}},\overline{\mathbb{Q}}_{\ell})\cong H_c^i(X_{\overline{s}}, R\psi\overline{\mathbb{Q}}_{\ell})\ ,\]
and
\[ \log \zeta^{\mathrm{ss}}(X_{\eta},s) = \sum_{r\geq 1} \sum_{x\in X_s(\mathbb{F}_{q^r})} \mathrm{tr}^{\mathrm{ss}}(\Phi_{q^r} | (R\psi\overline{\mathbb{Q}}_{\ell})_x) \frac{q^{-rs}}{r}\ .\]
\end{thm}

\begin{proof} The first statement follows from \cite{SGA7II}, XIII, Prop 2.1.9. For the second statement, it comes down to
\[ \sum_{i=0}^{2\dim X} (-1)^i \mathrm{tr}^{\mathrm{ss}}(\Phi_q^r | H_c^i(X_{\overline{s}}, R\psi\overline{\mathbb{Q}}_{\ell})) = \sum_{x\in X_s(\mathbb{F}_{q^r})} \mathrm{tr}^{\mathrm{ss}}(\Phi_{q^r} | (R\psi\overline{\mathbb{Q}}_{\ell})_x)\ ,\]
which follows from the version of the Lefschetz trace formula in \cite{HainesNgo}, Prop. 10, part (2).
\end{proof}

With these preparations, we deduce the following result.

\begin{thm}\label{CohoVanish} There is a canonical $G_{\mathbb{Q}_p}$-equivariant isomorphism
\[ H_c^i(\mathcal{M}_{\Gamma(p^n),\overline{\eta}},\overline{\mathbb{Q}}_{\ell})\longrightarrow H_c^i(\mathcal{M}_{\Gamma(p^n),\overline{s}}, R\psi\overline{\mathbb{Q}}_{\ell})\ .\]
In particular, the formula for the semisimple local factor from Theorem \ref{BadReduction} holds true.
\end{thm}

\begin{proof} We cannot apply Theorem \ref{BadReduction} to the scheme $X=\mathcal{M}_{\Gamma(p^n)}$ as its divisor at infinity is not \'{e}tale over $S=\mathrm{Spec}\ \mathbb{Z}[m^{-1}]$. There is the following way to circumvent this difficulty. We always have a canonical $G_{\mathbb{Q}_p}$-equivariant morphism
\[ H_c^i(X_{\overline{\eta}},\overline{\mathbb{Q}}_{\ell})\longrightarrow H_c^i(X_{\overline{s}}, R\psi\overline{\mathbb{Q}}_{\ell})\ .\]
To check that it is an isomorphism, we can forget about the Galois action.

We may also consider $X$ as a scheme over $S^{\prime}=\mathrm{Spec}\ \mathbb{Z}[m^{-1}][\zeta_{p^n}]$. Let $X^{\prime}$ be the normalization of $X\times_S S^{\prime}$. As $X$ is normal (since regular) and on the generic fibers, $(X\times_S S^{\prime})_{\eta}$ is a disjoint union of copies of $X$, parametrized by the primitive $p^n$-th roots of unity, it follows that $X^{\prime}/S^{\prime}$ is a disjoint union of copies of $X/S^{\prime}$. With Theorem \ref{Compactification}, it follows that we may use Theorem \ref{BadReduction} for $X^{\prime}$. Hence
\[
H_c^i(X_{\overline{\eta}},\overline{\mathbb{Q}}_{\ell})\cong H_c^i(X^{\prime}_{\overline{s}^{\prime}},R\psi^{\prime}\overline{\mathbb{Q}}_{\ell})\ ,
\]
where $R\psi^{\prime}\overline{\mathbb{Q}}_{\ell}$ are the nearby cycles for $X^{\prime}$ and $\overline{s}^{\prime}$ is the geometric special point of $S^{\prime}$. Note that $g: X_{\overline{s}^{\prime}}\longrightarrow X_{\overline{s}}$ is an infinitesimal thickening. Furthermore, the composite morphism $f: X^{\prime}\longrightarrow X\times_S S^{\prime}\longrightarrow X$ is finite and hence $g^{\ast} R\psi\overline{\mathbb{Q}}_{\ell}\cong f_{\overline{s}^{\prime}\ast} R\psi^{\prime}\overline{\mathbb{Q}}_{\ell}$. Therefore
\[
H_c^i(X^{\prime}_{\overline{s}^{\prime}},R\psi^{\prime}\overline{\mathbb{Q}}_{\ell})\cong H_c^i(X_{\overline{s}^{\prime}},f_{\overline{s}^{\prime}\ast}R\psi^{\prime}\overline{\mathbb{Q}}_{\ell})\cong H_c^i(X_{\overline{s}},R\psi\overline{\mathbb{Q}}_{\ell})\ .
\]
\end{proof}

\section{Calculation of the nearby cycles}

Again, let $X_{\mathcal{O}} / \mathcal{O}$ be a scheme of finite type. Let $X_{\eta^{\mathrm{ur}}}$ be the base-change of $X_{\mathcal{O}}$ to the maximal unramified extension $K^{\mathrm{ur}}$ of $K$ and let $X_{\mathcal{O}^{\mathrm{ur}}}$ be the base-change to the ring of integers in $K^{\mathrm{ur}}$. Then we have $\iota: X_{\overline{s}}\longrightarrow X_{\mathcal{O}^{\mathrm{ur}}}$ and $j: X_{\eta^{\mathrm{ur}}}\longrightarrow X_{\mathcal{O}^{\mathrm{ur}}}$.

\begin{lem}\label{InertiaNearby} In this setting
\[ R_{I_K}(R\psi \mathcal{F})=\iota^{\ast}Rj_{\ast} \mathcal{F}_{\eta^{\mathrm{ur}}}. \]
\end{lem}

\begin{proof} Both sides are the derived functors of the same functor.
\end{proof}

We now give a calculation in the case of interest to us.

\begin{thm}\label{CalcInertiaNearby} Let $X/ \mathcal{O}$ be regular and flat of relative dimension 1 and assume that $X_s$ is globally the union of regular divisors. Let $x\in X_s(\mathbb{F}_q)$ and let $D_1,\ldots,D_i$ be the divisors passing through $x$. Let $W_1$ be the $i$-dimensional $\overline{\mathbb{Q}}_{\ell}$-vector space with basis given by the $D_t$, $t=1,\ldots,i$, and let $W_2$ be the kernel of the map $W_1\longrightarrow\overline{\mathbb{Q}}_{\ell}$ sending all $D_t$ to 1. Then there are canonical isomorphisms
\[
(\iota^{\ast}R^kj_{\ast} \overline{\mathbb{Q}}_{\ell})_x\cong\left\{\begin{array}{ll} \overline{\mathbb{Q}}_{\ell} & k=0\\
W_1(-1) & k = 1\\
W_2(-2) & k = 2\\
0 & \mathrm{else}\ .\end{array} \right .
\]
\end{thm}

\begin{proof} We use a method similar to the one employed in \cite{RapoportZink}, pp.36-38. Denote $b_i: D_i\longrightarrow X$ and $b: x\longrightarrow X$ the closed embeddings. Let $I^{\bullet}$ be an injective resolution of $\overline{\mathbb{Q}}_{\ell}$ on $X$.

\begin{rem} Note that this is abuse of language, as in Remark \ref{DerivedStrange}. The proper meaning is to take a compatible system of injective resolutions of $\mathbb{Z}/\ell^n\mathbb{Z}$.
\end{rem}

Using diverse adjunction morphisms, we get a complex of sheaves on $X_{\mathcal{O}^{\mathrm{ur}}}$
\[\ldots \longrightarrow 0\longrightarrow W_2\otimes b_{\ast} b^{!} I^{\bullet}\longrightarrow \bigoplus\nolimits_i b_{i\ast} b_i^{!} I^{\bullet}\longrightarrow \iota^{\ast} I^{\bullet}\longrightarrow \iota^{\ast}j_{\ast}j^{\ast} I^{\bullet}\longrightarrow 0\ .\]

\begin{prop} The hypercohomology of this complex vanishes.
\end{prop}

\begin{proof} This is almost exactly \cite{RapoportZink}, Lemma 2.5. We repeat the argument here.

Let us begin with some general remarks. Recall that for a closed embedding $i: Y\longrightarrow Z$, there is a right-adjoint functor $i^{!}$ to $i_{!}=i_{\ast}$, given by $i^{!} \mathcal{F} = \mathrm{ker}(\mathcal{F}\longrightarrow j_{\ast}j^{\ast}\mathcal{F})$, where $j:Z\setminus Y\longrightarrow Z$ is the inclusion of the complement. We get an exact sequence
\[0\longrightarrow i_{!}i^{!} \mathcal{F}\longrightarrow \mathcal{F}\longrightarrow j_{\ast}j^{\ast}\mathcal{F}\ .\]
Being right-adjoint, $i^{!}$ is left-exact. Furthermore, $i_{\ast}$ has a left adjoint $i^{\ast}$ and a right adjoint $i^{!}$, hence is exact. Thus $i^{!}$ has an exact left adjoint and hence preserves injectives. Similarly, $i_{\ast}$ has the exact left adjoint $i^{\ast}$ and thus preserves injectives. We see that if $\mathcal{F}$ is injective, then $i_{!}i^{!}\mathcal{F} = i_{\ast} i^{!} \mathcal{F}$ is injective. Therefore $\mathcal{F} = i_{\ast} i^{!} \mathcal{F} \oplus \mathcal{F}^{\prime}$ for some injective sheaf $\mathcal{F}^{\prime}$. We get an injection $\mathcal{F}^{\prime}\longrightarrow j_{\ast}j^{\ast} \mathcal{F}$. Since $\mathcal{F}^{\prime}$ is injective, this is a split injection, with cokernel supported on $Y$. But
\[ \mathrm{Hom}(i_{\ast}\mathcal{G}, j_{\ast} j^{\ast}\mathcal{F}) = \mathrm{Hom}(j^{\ast}i_{\ast}\mathcal{G},j^{\ast}\mathcal{F}) = 0 \]
for any sheaf $\mathcal{G}$ on $Y$, hence the cokernel is trivial and $\mathcal{F}^{\prime} = j_{\ast}j^{\ast}\mathcal{F}$. This shows that
\[ \mathcal{F} = i_{\ast}i^{!} \mathcal{F} \oplus j_{\ast}j^{\ast}\mathcal{F} \]
for any injective sheaf $\mathcal{F}$ on $Y$, where $i^{!} \mathcal{F}$ and $j^{\ast}\mathcal{F}$ are injective, as $j^{\ast}=j^{!}$ has the exact left adjoint $j_{!}$.

We prove the proposition for any complex of injective sheaves $I^{\bullet}$. This reduces the problem to doing it for a single injective sheaf $I$. Let $U$ be the complement of $X_s$ in $X$ and let $U_i$ be the complement of $x$ in $D_i$. In our situation, we get a decomposition of $I$ as
\[ I = f_{U\ast}I_U \oplus \bigoplus\nolimits_i f_{U_i\ast} I_{U_i} \oplus f_{x\ast} I_x\ , \]
where $I_T$ is an injective sheaf on $T$ and $f_{T}: T\longrightarrow X$ is the locally closed embedding, for any $T$ occuring as an index.

Now we check case by case. First, $b^{!} f_{U\ast} I_U = b_{i}^{!} f_{U\ast} I_U = 0$ and the complex reduces to $\iota^{\ast} f_{U\ast} I_U\cong \iota^{\ast} f_{U\ast} I_U$. Second, $b^{!} f_{U_i\ast} I_{U_i}= b_{j}^{!} f_{U_i\ast} I_{U_i} = 0$ for $j\neq i$, while $b_{i}^{!} f_{U_i\ast} I_{U_i} = I_{U_i}$. Hence the complex reduces to the isomorphism $I_{U_i}\cong I_{U_i}$ in this case. In the last case, $b^{!} f_{x\ast} I_x = b_{i}^{!} f_{x\ast} I_x = f_{x\ast} I_x$ for all $i$, and hence the complex reduces to
\[ \ldots \longrightarrow 0\longrightarrow W_2\otimes f_{x\ast} I_x\longrightarrow W_1\otimes f_{x\ast} I_x\longrightarrow f_{x\ast} I_x\longrightarrow 0\ ,\]
and this is exact by our definition of $W_2$.
\end{proof}

Let us recall one important known special case of Grothendieck's purity conjecture:

\begin{thm} Let $X$ be a regular separated noetherian scheme of finite type over the ring of integers in a local field and let $f: Y\longrightarrow X$ be a closed immersion of a regular scheme $Y$ that is of codimension $d$ at each point. Let $\ell$ be a prime that is invertible on $X$. Then there is an isomorphism in the derived category of constructible $\overline{\mathbb{Q}}_{\ell}$-sheaves
\[ Rf^{!} \overline{\mathbb{Q}}_{\ell}\cong \overline{\mathbb{Q}}_{\ell}(-d)[-2d]\ .\]
\end{thm}

\begin{proof} This is contained in \cite{Thomason}, Cor. 3.9.
\end{proof}

We use this to get isomorphisms
\[ b_{i\ast} b_i^{!} I^{\bullet}\cong b_{i\ast} \overline{\mathbb{Q}}_{\ell}(-1)[-2]\ \]
and
\[ b_{\ast} b^{!} I^{\bullet}\cong b_{\ast} \overline{\mathbb{Q}}_{\ell}(-2)[-4] \]
in the derived category. Hence, since the spectral sequence for hypercohomology of
\[\ldots \longrightarrow 0\longrightarrow W_2\otimes b_{\ast} b^{!} I^{\bullet}\longrightarrow \bigoplus\nolimits_i b_{i\ast} b_i^{!} I^{\bullet}\longrightarrow \iota^{\ast}  I^{\bullet} \]
is equivariant for the Galois action and its only nonzero terms are of the form $\overline{\mathbb{Q}}_{\ell}(-k)$ for different $k$, it degenerates and we get the desired isomorphism.
\end{proof}

As a corollary, we can compute the semisimple trace of Frobenius on the nearby cycles in our situation. Let $B$ denote the Borel subgroup of $\mathrm{GL}_2$. Recall that we associated an element $\delta\in \mathrm{GL}_2(\mathbb{Q}_{p^r})$ to any point $x\in \mathcal{M}(\mathbb{F}_{p^r})$ by looking at the action of $F$ on the crystalline cohomology. We have the covering $\pi_n: \mathcal{M}_{\Gamma(p^n)}\longrightarrow \mathcal{M}$ and the sheaf $\mathcal{F}_n = \pi_{n\eta\ast} \mathbb{Q}_{\ell}$ on the generic fibre of $\mathcal{M}_{\Gamma(p^n)}$.

\begin{cor}\label{CalcSemisimpleTrace} Let $x\in \mathcal{M}(\mathbb{F}_{p^r})$ and let $g\in \mathrm{GL}_2(\mathbb{Z}_p)$.
\medskip

{\rm (i)} If $x$ corresponds to an ordinary elliptic curve and $a$ is the unique eigenvalue of $N\delta$ with valuation 0, then
\[ \mathrm{tr}^{\mathrm{ss}}(\Phi_{p^r}g | (R\psi \mathcal{F}_{n})_x) = \mathrm{tr}( \Phi_{p^r}g | V_n )\ , \]
where $V_n$ is a $G_{\mathbb{F}_{p^r}}\times \mathrm{GL}_2(\mathbb{Z}/p^n\mathbb{Z})$-representation isomorphic to
\[\bigoplus_{\chi\in ((\mathbb{Z}/p^n\mathbb{Z})^{\times})^{\vee}} \mathrm{Ind}_{B(\mathbb{Z}/p^n\mathbb{Z})}^{\mathrm{GL}_2(\mathbb{Z}/p^n\mathbb{Z})} 1\boxtimes \chi\]
as a $\mathrm{GL}_2(\mathbb{Z}/p^n\mathbb{Z})$-representation. Here $\Phi_{p^r}$ acts as the scalar $\chi(a)^{-1}$ on
\[
\mathrm{Ind}_{B(\mathbb{Z}/p^n\mathbb{Z})}^{\mathrm{GL}_2(\mathbb{Z}/p^n\mathbb{Z})} 1\boxtimes \chi\ .
\]
\medskip

{\rm (ii)} If $x$ corresponds to a supersingular elliptic curve, then
\[ \mathrm{tr}^{\mathrm{ss}}(\Phi_{p^r}g | (R\psi \mathcal{F}_{n})_x) = 1-\mathrm{tr}( g | \mathrm{St} )p^r\ , \]
where
\[ \mathrm{St} = \mathrm{ker}(\mathrm{Ind}_{B(\mathbb{Z}/p^n\mathbb{Z})}^{\mathrm{GL}_2(\mathbb{Z}/p^n\mathbb{Z})} 1\boxtimes 1\longrightarrow 1) \]
is the Steinberg representation of $\mathrm{GL}_2(\mathbb{Z}/p^n\mathbb{Z})$.
\end{cor}

\begin{proof} Note first that we have
\[ R_{I_K} R\psi_{\mathcal{M}} \mathcal{F}_n = R_{I_K} R\psi_{\mathcal{M}} \pi_{n\eta\ast} \overline{\mathbb{Q}}_{\ell} = \pi_{n\overline{s}\ast} R_{I_K} R\psi_{\mathcal{M}_{\Gamma(p^n)}} \overline{\mathbb{Q}}_{\ell} = \pi_{n\overline{s}\ast} \iota^{\ast} Rj_{\ast} \overline{\mathbb{Q}}_{\ell} \]
because $\pi_n$ is finite. Here subscripts for $R\psi$ indicate with respect to which scheme the nearby cycles are taken, and $\iota$ and $j$ are as defined before Lemma \ref{InertiaNearby}, for the scheme $\mathcal{M}_{\Gamma(p^n)}$. Note that we may apply Theorem \ref{CalcInertiaNearby} because of Theorem \ref{SpecialFiber}.

Let $\tilde{x}$ be any point above $x$ in $\mathcal{M}_{\Gamma(p^n)}$. In case (i), we see that
\[(\iota^{\ast} R^kj_{\ast} \overline{\mathbb{Q}}_{\ell})_{\tilde{x}}\cong \left\{\begin{array}{ll} \overline{\mathbb{Q}}_{\ell}(-k) & k=0,1\\ 0 & \mathrm{else}\ . \end{array} \right .\]
It remains to understand the action of $\mathrm{GL}_2(\mathbb{Z}/p^n\mathbb{Z})\times G_{\mathbb{F}_{p^r}}$ on $\pi_n^{-1}(x)$. Let $E$ be the elliptic curve corresponding to $x$. Fix an identification $E[p^{\infty}]\cong \mu_{p^{\infty}}\times \mathbb{Q}_p/\mathbb{Z}_p$ and in particular $E[p^n]\cong \mu_{p^n}\times \mathbb{Z}/p^n\mathbb{Z}$. Then the Drinfeld-level-$p^n$-structures are parametrized by surjections
\[
(\mathbb{Z}/p^n\mathbb{Z})^2\longrightarrow \!\!\!\!\!\!\!\!\! \longrightarrow \mathbb{Z}/p^n\mathbb{Z}\ .
\]
The right action of $\mathrm{GL}_2(\mathbb{Z}/p^n\mathbb{Z})$ is given by precomposition.

The identification $E[p^{\infty}]\cong \mu_{p^{\infty}}\times \mathbb{Q}_p/\mathbb{Z}_p$ gives $\delta$ the form
\[\left(\begin{array}{cc} pb_0 & 0 \\ 0 & a_0 \end{array}\right)\ ,\]
because the crystalline cohomology of $E$ agrees with the contravariant Dieu\-donn\'{e} module of $E[p^{\infty}]$. Then
\[\Phi_{p^r} = N\delta = \left(\begin{array}{cc} p^rNb_0 & 0 \\ 0 & Na_0 \end{array}\right)\ .\]
Hence $a=Na_0$ and $\Phi_{p^r}$ acts through multiplication by $a^{-1}$ on the factor $\mathbb{Z}/p^n\mathbb{Z}$ of $E[p^n]$. Hence it sends a Drinfeld-level-$p^n$-structure given by some surjection to the same surjection multiplied by $a^{-1}$. In total, we get
\[ (\pi_{n\overline{s}\ast} \iota^{\ast} R^kj_{\ast} \overline{\mathbb{Q}}_{\ell})_x\cong \left\{\begin{array}{ll} V_n(-k) & k = 0,1\\ 0 & \mathrm{else}\ , \end{array} \right . \]
where $V_n\cong \overline{\mathbb{Q}}_{\ell}^{M}$ with $M$ the set of surjections $(\mathbb{Z}/p^n\mathbb{Z})^2\longrightarrow \!\!\!\!\!\!\!\!\! \longrightarrow \mathbb{Z}/p^n\mathbb{Z}$. Then $\mathrm{GL}_2(\mathbb{Z}/p^n\mathbb{Z})\times G_{\mathbb{F}_{p^r}}$ acts on those pairs and hence on $V_n$, compatible with the action on the left hand side. The action of diagonal multiplication commutes with this action and gives rise to the decomposition
\[ V_n = \bigoplus_{\chi\in ((\mathbb{Z}/p^n\mathbb{Z})^{\times})^\vee} V_{\chi}\ . \]
Then one checks that $V_{\chi} = \mathrm{Ind}_{B(\mathbb{Z}/p^n\mathbb{Z})}^{\mathrm{GL}_2(\mathbb{Z}/p^n\mathbb{Z})} 1\boxtimes \chi$ and one easily arrives at the formula in case (i).

In case (ii), there is only one point $\tilde{x}$ above $x$ in $\mathcal{M}_{\Gamma(p^n)}$ and $p^n+p^{n-1}$ irreducible components meet at $x$, parametrized by $\mathbb{P}^1(\mathbb{Z}/p^n\mathbb{Z})$. This parametrization is $\mathrm{GL}_2(\mathbb{Z}/p^n\mathbb{Z})$-equivariant, so that
\[
W_1\cong \mathrm{Ind}_{B(\mathbb{Z}/p^n\mathbb{Z})}^{\mathrm{GL}_2(\mathbb{Z}/p^n\mathbb{Z})} 1\boxtimes 1
\]
and
\[
W_2\cong \ker(\mathrm{Ind}_{B(\mathbb{Z}/p^n\mathbb{Z})}^{\mathrm{GL}_2(\mathbb{Z}/p^n\mathbb{Z})} 1\boxtimes 1\longrightarrow 1)=\mathrm{St}
\]
in the notation of Theorem \ref{CalcInertiaNearby}. This yields the desired result.
\end{proof}

Define for $x\in \mathcal{M}(\mathbb{F}_{p^r})$
\[(R\psi \mathcal{F}_{\infty})_x = \lim_{\longrightarrow} (R\psi \mathcal{F}_n)_x\ .\]
It carries a natural smooth action of $\mathrm{GL}_2(\mathbb{Z}_p)$ and a commuting continuous action of $G_{\mathbb{Q}_{p^r}}$. Then we can define $\mathrm{tr}^{\mathrm{ss}}(\Phi_{p^r}h | (R\psi \mathcal{F}_{\infty})_x)$ for $h\in C_c^{\infty}(\mathrm{GL}_2(\mathbb{Z}_p))$ in the following way: Choose $n$ such that $h$ is $\Gamma(p^n)_{\mathbb{Q}_p}$-biinvariant and then take invariants under $\Gamma(p^n)_{\mathbb{Q}_p}$ first:
\[\mathrm{tr}^{\mathrm{ss}}(\Phi_{p^r} h|(R\psi \mathcal{F}_{\infty})_x) = \mathrm{tr}^{\mathrm{ss}}(\Phi_{p^r} h|(R\psi \mathcal{F}_{n})_x)\ .\]
It is easily checked that this gives something well-defined.

We see that for $h\in C_c^{\infty}(\mathrm{GL}_2(\mathbb{Z}_p))$, the value of $\mathrm{tr}^{\mathrm{ss}}(\Phi_{p^r}h | (R\psi \mathcal{F}_{\infty})_x)$ depends only on the element $\gamma=N\delta$ associated to $x$. This motivates the following definition.

\begin{definition} For $\gamma\in \mathrm{GL}_2(\mathbb{Q}_p)$ and $h\in C_c^{\infty}(\mathrm{GL}_2(\mathbb{Z}_p))$, define
\[
c_r(\gamma,h) = 0
\]
unless $v_p(\det \gamma)=r$, $v_p(\tr \gamma)\geq 0$. Assume now that these conditions are fulfilled. Then for $v_p(\tr \gamma)=0$, we define
\[
c_r(\gamma,h) = \sum_{\chi_0\in ((\mathbb{Z}/p^n\mathbb{Z})^{\times})^{\vee}} \tr(h|\mathrm{Ind}_{B(\mathbb{Z}/p^n\mathbb{Z})}^{\mathrm{GL}_2(\mathbb{Z}/p^n\mathbb{Z})} 1\boxtimes \chi_0) \chi_0(t_2)^{-1}
\]
where $t_2$ is the unique eigenvalue of $\gamma$ with $v_p(t_2)=0$. For $v_p(\tr \gamma)\geq 1$, we take
\[
c_r(\gamma,h) = \tr( h | 1 ) - p^r \tr( h | \mathrm{St} )\ .
\]
\end{definition}

Since $x$ is supersingular if and only if $\tr N\delta\equiv 0\modd p$, we get
\[
\mathrm{tr}^{\mathrm{ss}}(\Phi_{p^r}h | (R\psi \mathcal{F}_{\infty})_x) = c_r(N\delta, h)
\]
whenever $\delta$ is associated to $x\in \mathcal{M}(\mathbb{F}_{p^r})$.

\section{The semisimple trace of Frobenius as a twisted orbital integral}

First, we construct the function $\phi_p$ which will turn out to have the correct twisted orbital integrals.

\begin{lem}\label{FunctionExists} There is a function $\phi_p$ of the Bernstein center of $\mathrm{GL}_2(\mathbb{Q}_{p^r})$ such that for all irreducible smooth representations $\Pi$ of $\mathrm{GL}_2(\mathbb{Q}_{p^r})$, $\phi_p$ acts by the scalar
\[ p^{\frac 12 r} \mathrm{tr}^{\mathrm{ss}}(\Phi_{p^r}|\sigma_{\Pi})\ , \]
where $\sigma_{\Pi}$ is the representation of the Weil group $W_{\mathbb{Q}_{p^r}}$ of $\mathbb{Q}_{p^r}$ with values in $\overline{\mathbb{Q}}_{\ell}$ associated to $\Pi$ by the Local Langlands Correspondence.
\end{lem}

\begin{rem} Of course, the definition of the semisimple trace of Frobenius makes sense for representations of $W_{\mathbb{Q}_{p^r}}$. For a representation $\sigma$ of $W_{\mathbb{Q}_{p^r}}$, we write $\sigma^{\mathrm{ss}}$ for the associated semisimplification.
\end{rem}

\begin{proof} By Theorem \ref{BernsteinCenterIsom}, we only need to check that this defines a regular function on $D/W(L,D)$ for all $L$, $D$. First, note that the scalar agrees for a 1-dimensional representation $\Pi$ and the corresponding twist of the Steinberg representation, because we are taking the semisimple trace. This shows that we get a well-defined function on $D/W(L,D)$. But if one fixes $L$ and $D$ and takes $\Pi$ in the corresponding component, then the semi-simplification $\sigma_{\Pi}^{\mathrm{ss}}$ decomposes as $(\sigma_1\otimes \chi_1\circ \det)\oplus\cdots \oplus (\sigma_t\otimes \chi_t\circ \det)$ for certain fixed irreducible representations $\sigma_1,\ldots,\sigma_t$ and varying unramified characters $\chi_1,\ldots,\chi_t$ parametrized by $D$. In particular,
\[ \mathrm{tr}^{\mathrm{ss}}(\Phi_{p^r}|\sigma_{\Pi}) = \sum_{i=1}^t \mathrm{tr}^{\mathrm{ss}}(\Phi_{p^r}|\sigma_i)\chi_i(p) \]
which is clearly a regular function on $D$ and necessarily $W(L,D)$-invariant, hence descends to a regular function on $D/W(L,D)$.
\end{proof}

We also need the function $\phi_{p,0}=\phi_p\ast e_{\mathrm{GL}_2(\mathbb{Z}_{p^r})}\in \mathcal{H}(\mathrm{GL}_2(\mathbb{Q}_{p^r}),\mathrm{GL}_2(\mathbb{Z}_{p^r}))$. This definition is compatible with our previous use of $\phi_{p,0}$:

\begin{lem} The function $\phi_{p,0}$ is the characteristic function of the set
\[ \mathrm{GL}_2(\mathbb{Z}_{p^r})\left(\begin{array}{cc} p & 0\\ 0 & 1 \end{array}\right) \mathrm{GL}_2(\mathbb{Z}_{p^r})\]
divided by the volume of $\mathrm{GL}_2(\mathbb{Z}_{p^r})$.
\end{lem}

\begin{proof} Both functions are elements of the spherical Hecke algebra
\[
\mathcal{H}(\mathrm{GL}_2(\mathbb{Q}_{p^r}),\mathrm{GL}_2(\mathbb{Z}_{p^r}))\ .
\]
Since the Satake transform is an isomorphism, it suffices to check that the characteristic function of the given set divided by its volume acts through the scalars
\[
p^{\frac 12 r}\mathrm{tr}^{\mathrm{ss}}(\Phi_{p^r}|\sigma_{\Pi})
\]
on unramified representations. In general, this is done in \cite{KottwitzTO}, Theorem 2.1.3. Let us explain what it means here. By the Satake parametrization, an unramified representation $\Pi$ is given by two unramified characters $\chi_1$, $\chi_2$. Then $\sigma_{\Pi} = \chi_1\oplus \chi_2$ and
\[
\mathrm{tr}^{\mathrm{ss}}(\Phi_{p^r}|\sigma_{\Pi}) = \chi_1(p)+\chi_2(p)\ .
\]
Hence this is just the usual formula for the trace of the classical Hecke operators, usually called $T_p$, in terms of the Satake parameters (at least for $r=1$).
\end{proof}

\begin{thm}\label{MainTheorem} Let $\delta\in \mathrm{GL}_2(\mathbb{Q}_{p^r})$ with semisimple norm $\gamma\in \mathrm{GL}_2(\mathbb{Q}_p)$. Let $h\in C_c^{\infty}(\mathrm{GL}_2(\mathbb{Z}_p))$ and $h^{\prime}\in C_c^{\infty}(\mathrm{GL}_2(\mathbb{Z}_{p^r}))$ have matching (twisted) orbital integrals. Then
\[TO_{\delta\sigma}(\phi_p\ast h^{\prime}) = TO_{\delta\sigma}(\phi_{p,0}) c_r(\gamma,h) \ .\]
\end{thm}

\begin{proof} Let $f_1=\phi_p\ast h^{\prime}$ and let $f_2=\phi_{p,0}$. Let $\Pi$ be the base-change lift of some tempered representation $\pi$ of $\mathrm{GL}_2(\mathbb{Q}_p)$. Then, tracing through the definitions and taking $n$ so that $h$ and $h^{\prime}$ are $\Gamma(p^n)_{\mathbb{Q}_p}$ resp. $\Gamma(p^n)_{\mathbb{Q}_{p^r}}$-biinvariant,
\[\begin{aligned}
\tr ( (f_1,\sigma) | \Pi ) &= p^{\frac 12 r}\tr((h^{\prime},\sigma)|\Pi^{\Gamma(p^n)}) \mathrm{tr}^{\mathrm{ss}}(\Phi_{p^r}|\sigma_{\Pi})\\
& = p^{\frac 12 r}\tr(h|\pi^{\Gamma(p^n)}) \mathrm{tr}^{\mathrm{ss}}(\Phi_{p^r}|\sigma_{\Pi})
\end{aligned}\]
(because $h$ and $h^{\prime}$ have matching (twisted) orbital integrals) and
\[\tr ( (f_2,\sigma) | \Pi ) = p^{\frac 12 r}\dim \pi^{\mathrm{GL}_2(\mathbb{Z}_p)} \mathrm{tr}^{\mathrm{ss}}(\Phi_{p^r}|\sigma_{\Pi})\ ,\]
because $e_{\Gamma(1)_{\mathbb{Q}_p}}$ and $e_{\Gamma(1)_{\mathbb{Q}_{p^r}}}$ are associated by Theorem \ref{BaseChangeIdentity}.

As a first step, we prove the theorem for special $\delta$.

\begin{lem}\label{TORegularSplit} Assume that
\[
\delta = \left(\begin{array}{cc} t_1 & 0 \\ 0 & t_2 \end{array}\right)\ ,
\]
with $Nt_1\neq Nt_2$. Then the twisted orbital integrals
\[
TO_{\delta\sigma}(\phi_p\ast h^{\prime})=TO_{\delta\sigma}(\phi_{p,0})=0
\]
vanish except in the case where, up to exchanging $t_1$, $t_2$, we have $v_p(t_1)=1$ and $v_p(t_2)=0$. In the latter case,
\[
TO_{\delta\sigma}(\phi_p\ast h^{\prime}) = \mathrm{vol}(T(\mathbb{Z}_p))^{-1}\sum_{\chi_0\in ((\mathbb{Z}/p^n\mathbb{Z})^{\times})^{\vee}} \tr(h|\mathrm{Ind}_{B(\mathbb{Z}/p^n\mathbb{Z})}^{\mathrm{GL}_2(\mathbb{Z}/p^n\mathbb{Z})} 1\boxtimes \chi_0) \chi_0(Nt_2)^{-1}
\]
and
\[
TO_{\delta\sigma}(\phi_{p,0}) = \mathrm{vol}(T(\mathbb{Z}_p))^{-1}\ .
\]
\end{lem}

We remark that this implies the Theorem in this case.

\begin{proof} Let $B$ be the standard Borel subgroup consisting of upper triangular elements and let $\chi$ be a unitary character of $T(\mathbb{Q}_p)$ (and hence of $B(\mathbb{Q}_p)$). Take the normalized induction $\pi_{\chi}=\text{n-Ind}_{B(\mathbb{Q}_p)}^{\mathrm{GL}_2(\mathbb{Q}_p)} \chi$, an irreducible tempered representation of $\mathrm{GL}_2(\mathbb{Q}_p)$. Then, by \cite{LanglandsBCGL2}, Lemma 7.2, the character $\Theta_{\pi_{\chi}}$, a locally integrable function, is supported on the elements conjugate to an element of $T(\mathbb{Q}_p)$ and for $t=(t_1,t_2)\in T(\mathbb{Q}_p)$ regular,
\[\Theta_{\pi_{\chi}}(t) = \frac{\chi(t_1,t_2)+\chi(t_2,t_1)}{|\frac{t_1}{t_2}-2+\frac{t_2}{t_1}|^{\frac 12}}\ .\]
Let $\Pi_{\chi}$ be the base-change lift of $\pi_{\chi}$, with twisted character $\Theta_{\Pi_{\chi},\sigma}$. For $t\in T(\mathbb{Q}_p)$, define
\[
TO_t(f)=\left\{\begin{array}{ll} TO_{\tilde{t}\sigma}(f) & t=N\tilde{t}\ \text{for some}\ \tilde{t}\in T(\mathbb{Q}_{p^r}) \\
0 & \text{else}\ . \end{array}\right .
\]
This definition is independent of the choice of $\tilde{t}$ as all choices are $\sigma$-conjugate.
We get by the twisted version of Weyl's integration formula, cf. \cite{LanglandsBCGL2}, p.99, for any $f\in C_c^{\infty}(\mathrm{GL}_2(\mathbb{Q}_{p^r}))$
\[\begin{aligned}
\tr ( (f,\sigma) | \Pi_{\chi} ) &= \int_{\mathrm{GL}_2(\mathbb{Q}_{p^r})} f(g) \Theta_{\Pi_{\chi},\sigma}(g) dg\\
&=\frac 12 \int_{T(\mathbb{Q}_p)} |\frac{t_1}{t_2}-2+\frac{t_2}{t_1}| TO_{t}(f) \frac{\chi(t_1,t_2)+\chi(t_2,t_1)}{|\frac{t_1}{t_2}-2+\frac{t_2}{t_1}|^{\frac 12}} dt\\
&=\int_{T(\mathbb{Q}_p)} |\frac{t_1}{t_2}-2+\frac{t_2}{t_1}|^{\frac 12} TO_{t}(f) \chi(t) dt\ ,
\end{aligned}\]
By Fourier inversion, we arrive at
\[ TO_{t}(f) = |\frac{t_1}{t_2}-2+\frac{t_2}{t_1}|^{-\frac 12} \int_{\widehat{T(\mathbb{Q}_p)}_u} \tr ( (f,\sigma) | \Pi_{\chi} ) \chi(t)^{-1} d\chi\ , \]
where $\widehat{T(\mathbb{Q}_p)}_u$ denotes the set of unitary characters of $T(\mathbb{Q}_p)$. Measures need to be chosen so that
\[
\mathrm{vol}(\widehat{T(\mathbb{Q}_p)}^0_u) = \mathrm{vol}(T(\mathbb{Z}_p))^{-1}\ ,
\]
where $\widehat{T(\mathbb{Q}_p)}^0_u$ is the identity component of $\widehat{T(\mathbb{Q}_p)}_u$; it consists precisely of the unramified characters.

Note that $TO_{t}(f)$ is a locally constant function on the set of regular elements of $T(\mathbb{Q}_p)$ and hence this gives an identity of functions there. From here, it is immediate that $TO_{t}(f_2)=0$ and $TO_{t}(f_1)=0$ for all $t=(t_1,t_2)$ with $t_1\neq t_2$, except in the case where (up to exchanging $t_1$,$t_2$), $v_p(t_1)=r$, $v_p(t_2)=0$. In the latter case, $TO_{t}(f_2)=\mathrm{vol}(T(\mathbb{Z}_p))^{-1}$. The calculation of $TO_{t}(f_1)$ is slightly more involved:
\[\begin{aligned}
TO_{t}(f_1) &= p^{-\frac 12 r}\int_{\widehat{T(\mathbb{Q}_p)}_u} \tr ( (f_1,\sigma) | \Pi_{\chi} ) \chi(t)^{-1}d\chi\\
&=\mathrm{vol}(T(\mathbb{Z}_p))^{-1} \sum_{\chi_0\in ((\mathbb{Z}/p^n\mathbb{Z})^{\times})^{\vee}} \tr(h|\mathrm{Ind}_{B(\mathbb{Z}/p^n\mathbb{Z})}^{\mathrm{GL}_2(\mathbb{Z}/p^n\mathbb{Z})} 1\boxtimes \chi_0) \chi_0(t_2)^{-1}\ ,
\end{aligned}\]
giving the desired result.
\end{proof}

Next, we remark that if $\delta$ is not $\sigma$-conjugate to an element as in Lemma \ref{TORegularSplit}, then the eigenvalues of $N\delta$ have the same valuation. Let
\[ f=f_1 + (H_1 p^r-H_2) f_2\ ,\]
where we have set
\[ H_1 = \tr ( h | \mathrm{St} )\ ,\ H_2 = \tr( h | 1 )\ ,\]
with $\mathrm{St}$ resp. $1$ the Steinberg resp. trivial representation of $\mathrm{GL}_2(\mathbb{Z}/p^n\mathbb{Z})$. Then the next lemma finishes the proof of the theorem.

\begin{lem} Assume that the eigenvalues of $N\delta$ have the same valuation. Then the twisted orbital integral $TO_{\delta\sigma}(f)$ vanishes.
\end{lem}

\begin{proof} Let $V$ be the set of all $\delta\in \mathrm{GL}_2(\mathbb{Q}_{p^r})$ such that the eigenvalues of $N\delta$ have the same valuation. Note that $V$ is open. In particular, its characteristic function $\chi_V$ is locally constant and hence $\tilde{f}(g)=f(g) \chi_V(g)$ defines a function $\tilde{f}\in C_c^{\infty}(\mathrm{GL}_2(\mathbb{Q}_{p^r}))$. Then, obviously, the twisted orbital integrals of $\tilde{f}$ and $f$ agree on all elements $\delta$ such that the eigenvalues of $N\delta$ have the same valuation. We will prove that for all tempered irreducible smooth representations $\pi$ of $\mathrm{GL}_2(\mathbb{Q}_p)$ with base-change lift $\Pi$, we have
\[
\tr((\tilde{f},\sigma)|\Pi)=0\ .
\]
By the usual arguments (cf. proof of Theorem \ref{BaseChangeIdentity}), this implies that all twisted orbital integrals of $\tilde{f}$ for elements $\delta$ with $N\delta$ semisimple vanish. This then proves the lemma.

First, we find another expression for $\tr((\tilde{f},\sigma)|\Pi)$. Note that we have seen in Lemma \ref{TORegularSplit} that the twisted orbital integrals of $f_i$ vanish on all elements of $\delta\in T$ with $N\delta$ having distinct eigenvalues of the same valuation, whence the same is true for $f$. In particular,
\[ \tr ( (\tilde{f},\sigma) | \Pi ) = \tr( (f,\sigma) | \Pi )_{\sigma\mathrm{-ell}} = \int_{\mathrm{GL}_2(\mathbb{Q}_{p^r})_{\sigma\mathrm{-ell}}} f(g) \Theta_{\Pi,\sigma}(g) dg\ ,\]
where $\mathrm{GL}_2(\mathbb{Q}_{p^r})_{\sigma\mathrm{-ell}}$ is the set of elements of $\delta\in \mathrm{GL}_2(\mathbb{Q}_{p^r})$ with $N\delta$ elliptic (since the character $\Theta_{\Pi,\sigma}$ is locally integrable, one could always restrict the integration to regular semisimple elements and hence non-semisimple elements need not be considered). This reduces us to proving that
\[
\tr((f,\sigma) | \Pi)_{\sigma\mathrm{-ell}} = 0\ .
\]

But for $\pi = \text{n-Ind}_{B(\mathbb{Q}_p)}^{\mathrm{GL}_2(\mathbb{Q}_p)} \chi$ the normalized induction of a unitary character, with base-change lift $\Pi$, we have
\[ \tr ( (f,\sigma) | \Pi )_{\sigma\mathrm{-ell}} = 0 \]
because the character $\Theta_{\pi}$ is supported in elements conjugate to an element of $T(\mathbb{Q}_{p})$. For $\pi$ supercuspidal with base-change lift $\Pi$,
\[ \tr ( (f,\sigma) | \Pi )_{\sigma\mathrm{-ell}} = \tr ( (f,\sigma) | \Pi ) = 0\ . \]
Here the second equation follows from the definitions of $f$ and the way $(f_i,\sigma)$ acts on $\Pi$, whereas the first equation holds because the character $\Theta_{\pi}$ is supported in the elements whose eigenvalues have the same valuation -- this easily follows from the fact that $\pi$ is compactly induced from a representation of an open subgroup that is compact modulo center, as proved in \cite{BushnellKutzko}. This leaves us with checking that
\[ \tr ( (f,\sigma) | \Pi )_{\sigma\mathrm{-ell}} = 0 \]
for any unitary twist of the Steinberg representation $\pi$ with base-change lift $\Pi$. However, restricted to the elliptic elements, the character of a twist of the Steinberg representation agrees up to sign with the character of the corresponding 1-dimensional representation. Hence it is enough to check that
\[ \tr ( (f,\sigma) | \Pi )_{\sigma\mathrm{-ell}} = 0 \]
for any a 1-dimensional representation $\pi = \chi \circ \det$ with base-change lift $\Pi = \chi\circ \mathrm{Norm}_{\mathbb{Q}_{p^r}/\mathbb{Q}_p}\circ \det$.

Then
\[\begin{aligned}
\tr ( (f,\sigma) | \Pi )_{\sigma\mathrm{-ell}} &= \tr ( (f,\sigma) | \Pi )\\
& - \frac 12 \int_{T(\mathbb{Q}_p)} |\frac{t_1}{t_2}-2+\frac{t_2}{t_1}|TO_{t}(f)\chi(t_1t_2) dt\ .
\end{aligned}\]
Note that the function in the integral only takes nonzero values if $v_p(t_1)=r$ and $v_p(t_2)=0$, or the other way around. Hence, we may rewrite the equality as
\[ \tr ( (f,\sigma) | \Pi )_{\sigma\mathrm{-ell}} = \tr ( (f,\sigma) | \Pi ) - p^r\chi(p^r) \int_{T(\mathbb{Z}_p)} TO_{(p^rt_1,t_2)}(f)\chi(t_1t_2) dt\ .\]
Now, if $\chi$ is ramified, then $\tr ( (f,\sigma) | \Pi ) = 0$, while the integral is zero as well, because $TO_{(p^rt_1,t_2)}(f)$ does not depend on $t_1$ and hence keeping $t_2$ fixed and integrating over $t_1$ gives zero.

On the other hand, if $\chi$ is unramified, then
\[\begin{aligned}
\tr ( (f,\sigma) | \Pi ) &= \tr ( (f_1,\sigma) | \Pi ) + (H_1 p^r-H_2) \tr ( (f_2,\sigma) | \Pi )\\
&= (1+p^r)\chi(p^r)H_2 + (H_1 p^r-H_2)(1+p^r)\chi(p^r)\\
&= (1+p^r)H_1 p^r\chi(p^r)
\end{aligned}\]
and the integral gives
\[\begin{aligned}
\int_{T(\mathbb{Z}_p)} TO_{(p^rt_1,t_2)}(f) dt &=\int_{T(\mathbb{Z}_p)} TO_{(p^rt_1,t_2)}(f_1) dt\\
& +(H_1 p^r-H_2)\int_{T(\mathbb{Z}_p)} TO_{(p^rt_1,t_2)}(f_2) dt\\
&=(H_1+H_2) + (H_1 p^r-H_2) = (1+p^r)H_1\ .
\end{aligned}\]
Putting everything together, we get the conclusion.
\end{proof}

\end{proof}

We get the following corollary.

\begin{cor}\label{MainCorollary} Let $x\in \mathcal{M}(\mathbb{F}_{p^r})$ with associated $\delta$. Let $h\in C_c^{\infty}(\mathrm{GL}_2(\mathbb{Z}_p))$ and $h^{\prime}\in C_c^{\infty}(\mathrm{GL}_2(\mathbb{Z}_{p^r}))$ have matching (twisted) orbital integrals. Then
\[TO_{\delta\sigma}(\phi_p\ast h^{\prime}) = TO_{\delta\sigma}(\phi_{p,0}) \mathrm{tr}^{\mathrm{ss}}(\Phi_{p^r} h|(R\psi \mathcal{F}_{\infty})_x) \ .\]
\end{cor}

\begin{proof} Combining Theorem \ref{MainTheorem} and Corollary \ref{CalcSemisimpleTrace}, all we have to check is the following lemma.

\begin{lem}\label{FrobSemisimple} For any $\delta\in \mathrm{GL}_2(\mathbb{Q}_{p^r})$ associated to an elliptic curve over $\mathbb{F}_{p^r}$, the norm $N\delta$ is semisimple.
\end{lem}

\begin{proof} As $N\delta$ is the endomorphism of crystalline cohomology associated to the geometric Frobenius $\Phi_{p^r}$ of $E_0$, it is enough to prove that any $\mathbb{F}_{p^r}$-self-isogeny $f: E\longrightarrow E$ of an elliptic curve $E/\mathbb{F}_{p^r}$ gives rise to a semisimple endomorphism on the crystalline cohomology. If not, we may find $m,n\in \mathbb{Z}$ such that $f^{\prime}=mf-n$ is nilpotent on crystalline cohomology, but nonzero. But if $f^{\prime}$ is nonzero, then for the dual isogeny $(f^{\prime})^{\ast}$, the composition $f^{\prime}(f^{\prime})^{\ast}$ is a scalar, and hence induces multiplication by a scalar on the crystalline cohomology. Hence $f^{\prime}$ induces an invertible endomorphism on the (rational) crystalline cohomology, contradiction.
\end{proof}

\end{proof}

We note that by Theorem \ref{BaseChangeIdentity}, we can take $h$ to be the idempotent $e_{\Gamma(p^n)_{\mathbb{Q}_p}}$ and $h^{\prime}$ to be the idempotent $e_{\Gamma(p^n)_{\mathbb{Q}_{p^r}}}$ and get the following corollary, proving Theorem B for $\phi_{p}\ast e_{\Gamma(p^n)_{\mathbb{Q}_{p^r}}}$ instead of $\phi_{p,n}$. For the comparison of these functions, we refer to Section 14.

\begin{cor}\label{corB} Let $x\in \mathcal{M}(\mathbb{F}_{p^r})$ with associated $\delta$. Then
\[\mathrm{tr}^{\mathrm{ss}}(\Phi_{p^r}|(R\psi \mathcal{F}_n)_x) = TO_{\delta\sigma}(\phi_p\ast e_{\Gamma(p^n)_{\mathbb{Q}_{p^r}}})(TO_{\delta\sigma}(\phi_{p,0}))^{-1}\ .\]
\end{cor}

\begin{proof} Use Remark \ref{NonvanishingTO} to see that the right-hand side is well-defined.
\end{proof}

\section{The Langlands-Kottwitz approach: Case of bad reduction}\label{LKApproach}

By Theorem \ref{CohoVanish}, we get
\[ \log \zeta^{\mathrm{ss}}(\mathcal{M}_{\Gamma(p^n)},\overline{\mathbb{Q}}_{\ell}) = \sum_{r\geq 1} \sum_{x\in \mathcal{M}(\mathbb{F}_{p^r})} \mathrm{tr}^{\mathrm{ss}}(\Phi_{p^r} | (R\psi\mathcal{F}_n)_x) \frac{p^{-rs}}{r}\ .\]
Again, we may split the terms according to their $\mathbb{F}_{p^r}$-isogeny class. This leads us to consider, for $E_0$ fixed,
\[ \sum_{x\in \mathcal{M}(\mathbb{F}_{p^r})(E_0)} \mathrm{tr}^{\mathrm{ss}}(\Phi_{p^r} | (R\psi\mathcal{F}_n)_x) \ .\]
Now Corollary \ref{corB} tells us that
\[ \mathrm{tr}^{\mathrm{ss}}(\Phi_{p^r} | (R\psi\mathcal{F}_n)_x) = TO_{\delta\sigma}(\phi_p\ast e_{\Gamma(p^n)_{\mathbb{Q}_{p^r}}})(TO_{\delta\sigma}(\phi_{p,0}))^{-1}\ . \]

\begin{cor}\label{LKNew} The sum
\[ \sum_{x\in \mathcal{M}(\mathbb{F}_{p^r})(E_0)} \mathrm{tr}^{\mathrm{ss}}(\Phi_{p^r} | (R\psi\mathcal{F}_n)_x)\]
equals
\[ \mathrm{vol}(\Gamma\backslash G_{\gamma}(\mathbb{A}_f^p)\times G_{\delta\sigma}(\mathbb{Q}_p)) O_{\gamma}(f^p) TO_{\delta\sigma}(\phi_p\ast e_{\Gamma(p^n)_{\mathbb{Q}_{p^r}}})\ .\]
\end{cor}

\begin{proof} This is obvious from what was already said and Theorem \ref{LKClassical}.
\end{proof}

First, we eliminate the twisted orbital integral. Let $f_{p,r}$ be the function of the Bernstein center for $\mathrm{GL}_2(\mathbb{Q}_p)$ such that for all irreducible smooth representations $\pi$ of $\mathrm{GL}_2(\mathbb{Q}_p)$, $f_{p,r}$ acts by the scalar
\[ p^{\frac 12 r} \mathrm{tr}^{\mathrm{ss}}(\Phi_p^r|\sigma_{\pi})\ , \]
where, again, $\sigma_{\pi}$ is the associated representation of the Weil group $W_{\mathbb{Q}_p}$ over $\overline{\mathbb{Q}}_{\ell}$. The existence of $f_{p,r}$ is proved in the same way as Lemma \ref{FunctionExists}. By \cite{HarrisTaylor}, we know that if $\pi$ is tempered and $\Pi$ is a base-change lift of $\pi$, then $\sigma_{\Pi}$ is the restriction of $\sigma_{\pi}$. Perhaps it is worth remarking that the statement on the semisimple trace of Frobenius that we need is much simpler.

\begin{lem} For any tempered irreducible smooth representation $\pi$ of $\mathrm{GL}_2(\mathbb{Q}_p)$ with base-change lift $\Pi$, we have
\[
\mathrm{tr}^{\mathrm{ss}} ( \Phi_p^r | \sigma_{\pi} ) = \mathrm{tr}^{\mathrm{ss}} ( \Phi_{p^r} | \sigma_{\Pi} )\ .
\]
\end{lem}

\begin{proof} Assume first that $\mathrm{tr}^{\mathrm{ss}} ( \Phi_{p^r} | \sigma_{\Pi} )\neq 0$. Then the semisimplification of $\sigma_{\Pi}$ is a sum of two characters $\chi_1$ and $\chi_2$, one of which, say $\chi_2$, is unramified and in particular invariant under $\mathrm{Gal}(\mathbb{Q}_{p^r}/\mathbb{Q}_p)$. Because $\Pi$ is invariant under the Galois group $\mathrm{Gal}(\mathbb{Q}_{p^r}/\mathbb{Q}_p)$, the character $\chi_1$ needs to factor over the norm map. We see that there is a principal series representation $\pi^{\prime}$ with base-change lift $\Pi$. By the uniqueness properties of base-change, cf. \cite{LanglandsBCGL2}, $\pi$ is also a principal series representation. The claim then follows from the explicit description of base-change for principal series representation.

Now assume $\mathrm{tr}^{\mathrm{ss}} ( \Phi_{p^r} | \sigma_{\Pi} )=0$. If $\mathrm{tr}^{\mathrm{ss}} ( \Phi_p^r | \sigma_{\pi} )\neq 0$, then the semisimplification of $\sigma_{\pi}$ is a sum of two characters (one of which is unramified), whence $\pi$ is again a principal series representation. This yields the claim as before.
\end{proof}

This shows that $f_{p,r}\ast e_{\Gamma(p^n)_{\mathbb{Q}_p}}$ and $\phi_p\ast e_{\Gamma(p^n)_{\mathbb{Q}_{p^r}}}$ satisfy the hypothesis of Theorem \ref{BaseChangeIdentity}. Thus, by Lemma \ref{FrobSemisimple}, we may rewrite the expression in Corollary \ref{LKNew} as

\begin{cor}\label{LKNewBC} The sum
\[ \sum_{x\in \mathcal{M}(\mathbb{F}_{p^r})(E_0)} \mathrm{tr}^{\mathrm{ss}}(\Phi_{p^r} | (R\psi\mathcal{F}_n)_x)\]
equals
\begin{equation}\label{eqn1} \pm \mathrm{vol}(\Gamma\backslash G_{\gamma}(\mathbb{A}_f^p)\times G_{\delta\sigma}(\mathbb{Q}_p)) O_{\gamma}(f^p) O_{N\delta}(f_{p,r}\ast e_{\Gamma(p^n)_{\mathbb{Q}_p}})\ .\end{equation}
\end{cor}

We need to recall certain facts from Honda-Tate theory to simplify our expression further.

\begin{thm}\label{HondaTate} Fix a finite field $\mathbb{F}_q$ of characteristic $p$.
\begin{enumerate}
\item[{\rm (a)}] For any elliptic curve $E/\mathbb{F}_q$, the action of Frobenius on $H^1_{\mathrm{et}}(E,\mathbb{Q}_{\ell})$ is semisimple with characteristic polynomial $p_E\in \mathbb{Z}[T]$ independent of $\ell$. Additionally, if $F$ acts as $\delta\sigma$ on $H^1_{\mathrm{cris}}(E/\mathbb{Z}_q)\otimes \mathbb{Q}_q$, then $N\delta$ is semisimple with characteristic polynomial $p_E$.
\end{enumerate}
Let $\gamma_E\in \mathrm{GL}_2(\mathbb{Q})$ be semisimple with characteristic polynomial $p_E$. Then
\begin{enumerate}
\item[{\rm (b)}] The map $E\longmapsto \gamma_E$ gives a bijection between $\mathbb{F}_q$-isogeny classes of elliptic curves over $\mathbb{F}_q$ and conjugacy classes of semisimple elements $\gamma\in \mathrm{GL}_2(\mathbb{Q})$ with $\det \gamma=q$ and $\tr \gamma\in \mathbb{Z}$ which are elliptic in $\mathrm{GL}_2(\mathbb{R})$.

\item[{\rm (c)}] Let $G_{\gamma_E}$ be the centralizer of $\gamma_E$. Then $\mathrm{End}(E)^{\times}$ is an inner form of $G_{\gamma_E}$. In fact,
\[\begin{aligned}
(\mathrm{End}(E)\otimes\mathbb{Q}_{\ell})^{\times}&\cong G_{\gamma_E}\otimes \mathbb{Q}_{\ell}\ ,\ \mathrm{for}\ \ell\neq p\\
(\mathrm{End}(E)\otimes\mathbb{Q}_p)^{\times}&\cong G_{\delta\sigma}\ .
\end{aligned}\]
Furthermore, $(\mathrm{End}(E)\otimes \mathbb{R})^{\times}$ is anisotropic modulo center.
\end{enumerate}\hfill$\Box$
\end{thm}

\begin{proof} This combines the fixed point formulas in \'{e}tale and crystalline cohomology, the Weil conjectures (here Weil's theorem) for elliptic curves and the main theorems of \cite{Tate}, \cite{Honda}.
\end{proof}

Regarding our expression for one isogeny class, we first get that \eqref{eqn1} equals
\begin{equation}\label{eqn2}
\pm \mathrm{vol}(\Gamma \backslash (\mathrm{End}(E)\otimes \mathbb{A}_f)^{\times}) O_{\gamma}(f^p) O_{\gamma}(f_{p,r}\ast e_{\Gamma(p^n)_{\mathbb{Q}_p}})\ ,\end{equation}
writing $\gamma=\gamma_E\in \mathrm{GL}_2(\mathbb{Q})$ as in the Theorem and using that by part (a), this is compatible with our previous use. Define the function
\[
f=f^p(f_{p,r}\ast e_{\Gamma(p^n)_{\mathbb{Q}_p}})\in C_c^{\infty}(\mathrm{GL}_2(\mathbb{A}_f))\ .
\]
Recalling that $\Gamma = (\mathrm{End}(E)\otimes \mathbb{Q})^{\times}$, we see that \eqref{eqn2} equals
\begin{equation}\label{eqn3}
\pm \mathrm{vol}((\mathrm{End}(E)\otimes \mathbb{Q})^{\times} \backslash (\mathrm{End}(E)\otimes \mathbb{A}_f)^{\times}) \int_{G_{\gamma}(\mathbb{A}_f)\backslash \mathrm{GL}_2(\mathbb{A}_f)} f(g^{-1}\gamma g) dg\ .
\end{equation}

For any reductive group $G$ over $\mathbb{Q}$, let $\overline{G}$ be any inner form of $G$ over $\mathbb{Q}$ which is anisotropic modulo center over $\mathbb{R}$, if existent. The terms where $\overline{G}$ occurs will not depend on the choice made because of the invariance of the Tamagawa number under inner twists. Collecting everything so far, we see that

\begin{thm}\label{TraceFinite} The Lefschetz number
\[ \sum_{x\in \mathcal{M}_{\Gamma(p^n)}(\mathbb{F}_{p^r})} \mathrm{tr}^{\mathrm{ss}}(\Phi_{p^r} | (R\psi\mathcal{F}_n)_x) \]
equals
\[\begin{aligned}
- & \sum_{\gamma\in Z(\mathbb{Q})} \mathrm{vol}(\overline{\mathrm{GL}}_2(\mathbb{Q})\backslash \overline{\mathrm{GL}}_2(\mathbb{A}_f)) f(\gamma)\\
+ & \sum_{\substack{\gamma\in \mathrm{GL}_2(\mathbb{Q})\setminus Z(\mathbb{Q}) \\ \mathrm{\rm semisimple\ conj.\ class}\\ \mathrm{\rm with}\ \gamma_{\infty}\ \mathrm{\rm elliptic}}} \mathrm{vol}(\overline{G}_{\gamma}(\mathbb{Q})\backslash \overline{G}_{\gamma}(\mathbb{A}_f)) \int_{G_{\gamma}(\mathbb{A}_f)\backslash \mathrm{GL}_2(\mathbb{A}_f)} f(g^{-1}\gamma g) dg\ .
\end{aligned}\]
\end{thm}

\begin{rem} If $\gamma\in \mathrm{GL}_2(\mathbb{Q})\setminus Z(\mathbb{Q})$ is semisimple with $\gamma_\infty$ elliptic, then $G_{\gamma}$ is already anisotropic modulo center over $\mathbb{R}$, so that one may take $\overline{G}_{\gamma} = G_{\gamma}$ in this case. We will not need this fact.
\end{rem}

\begin{proof} We only need to check that the contributions of $\gamma$ with $\det \gamma\neq p^r$ or $\tr \gamma\not\in\mathbb{Z}$ vanish. Assume that $\det \gamma\neq p^r$. The orbital integrals of $f^p$ vanish except if the determinant is a unit away from $p$, so that $\det \gamma$ is up to sign a power of $p$. The orbital integrals of $f_{p,r} \ast e_{\Gamma(p^n)_{\mathbb{Q}_p}}$ vanish except if $v_p(\det \gamma)=r$, so that $\det \gamma=\pm p^r$. But if $\det \gamma=-p^r<0$, then $\gamma$ is hyperbolic at $\infty$, contradiction.

Assume now that $\tr \gamma\not\in\mathbb{Z}$. The orbital integrals of $f^p$ vanish as soon as a prime $\ell\neq p$ is in the denominator of $\tr \gamma$. The orbital integrals of $f_{p,r} \ast e_{\Gamma(p^n)_{\mathbb{Q}_p}}$ match with the twisted orbital integrals of $\phi_p \ast e_{\Gamma(p^n)_{\mathbb{Q}_{p^r}}}$, which were computed in Theorem \ref{MainTheorem}. In particular, they are nonzero only if $v_p(\tr \gamma)\geq 0$, so that $\tr \gamma$ is necessarily integral.
\end{proof}

It turns out that it is easier to apply the Arthur-Selberg trace formula for the cohomology of the compactification $\overline{\mathcal{M}}_{p^n m}$ instead of the cohomology with compact supports of $\mathcal{M}_{p^n m}$. The corresponding modifications are done in the next section.

\section{Contributions from infinity}

Recall that the smooth curve $\mathcal{M}_{p^n m}/\mathrm{Spec}\ \mathbb{Z}[\frac 1{pm}]$ has a smooth projective compactification $j: \mathcal{M}_{p^n m}\longrightarrow \overline{\mathcal{M}}_{p^n m}$ with boundary $\partial \mathcal{M}_{p^n m}$. We use a subscript $\overline{\mathbb{Q}}$ to denote base change to $\overline{\mathbb{Q}}$. We are interested in the cohomology groups
\[
H^i(\overline{\mathcal{M}}_{p^n m \overline{\mathbb{Q}}},\overline{\mathbb{Q}}_{\ell})\ .
\]
Let
\[\begin{aligned}
H^{\ast}(\overline{\mathcal{M}}_{p^n m \overline{\mathbb{Q}}},\overline{\mathbb{Q}}_{\ell}) &= \sum_{i=0}^2 (-1)^i H^i(\overline{\mathcal{M}}_{p^n m \overline{\mathbb{Q}}},\overline{\mathbb{Q}}_{\ell})\ ,\\
H_c^{\ast}(\mathcal{M}_{p^n m \overline{\mathbb{Q}}},\overline{\mathbb{Q}}_{\ell}) &= \sum_{i=0}^2 (-1)^i H_c^i(\mathcal{M}_{p^n m \overline{\mathbb{Q}}},\overline{\mathbb{Q}}_{\ell})
\end{aligned}\]
in the Grothendieck group of representations of $G_{\mathbb{Q}}\times \mathrm{GL}_2(\mathbb{Z}/p^n m\mathbb{Z})$.
Then the long exact cohomology sequence for
\[
0\longrightarrow j_{!} \overline{\mathbb{Q}}_{\ell}\longrightarrow \overline{\mathbb{Q}}_{\ell}\longrightarrow \bigoplus_{x\in \partial \mathcal{M}_{p^n m \overline{\mathbb{Q}}}} \overline{\mathbb{Q}}_{\ell,x}\longrightarrow 0
\]
implies that
\[
H^{\ast}(\overline{\mathcal{M}}_{p^n m \overline{\mathbb{Q}}},\overline{\mathbb{Q}}_{\ell}) = H_c^{\ast}(\mathcal{M}_{p^n m \overline{\mathbb{Q}}},\overline{\mathbb{Q}}_{\ell}) + H^0(\partial \mathcal{M}_{p^n m \overline{\mathbb{Q}}},\overline{\mathbb{Q}}_{\ell})\ .
\]

\begin{lem}\label{PointsAtInfinity} There is a $G_{\mathbb{Q}}\times \mathrm{GL}_2(\mathbb{Z}/p^n m\mathbb{Z})$-equivariant bijection
\[
\partial \mathcal{M}_{p^n m\overline{\mathbb{Q}}}\cong \{\pm \left(\begin{array}{cc} 1 & \ast \\ 0 & 1 \end{array}\right) \}\backslash \mathrm{GL}_2(\mathbb{Z}/p^n m\mathbb{Z})\ ,
\]
where $\mathrm{GL}_2(\mathbb{Z}/p^n m\mathbb{Z})$ acts on the right hand side by multiplication from the right, and $G_{\mathbb{Q}}$ acts on the right hand side by multiplication from the left through the map
\[
G_{\mathbb{Q}}\longrightarrow \mathrm{Gal}(\mathbb{Q}(\zeta_{p^n m})/\mathbb{Q})\cong (\mathbb{Z}/p^n m\mathbb{Z})^{\times}\longrightarrow \mathrm{GL}_2(\mathbb{Z}/p^n m\mathbb{Z})\ ,
\]
the last map being given by
\[
x\longmapsto \left(\begin{array}{cc} x^{-1} & 0 \\ 0 & 1\end{array}\right)\ .
\]
\end{lem}

\begin{proof} This is contained in \cite{KatzMazur}. The point is that the points at infinity correspond to level-$p^nm$-structures on the rational $p^nm$-gon, cf. also \cite{DeligneRapoport}, and the automorphism group of the $p^nm$-gon is isomorphic to
\[
\{\pm \left(\begin{array}{cc} 1 & \ast \\ 0 & 1 \end{array}\right) \}\ .
\]
The Galois group acts on the $p^nm$-torsion points of the rational $p^nm$-gon only by its action on the $p^nm$-th roots of unity.
\end{proof}

We get the following corollary.

\begin{cor}\label{TraceInfinity} The semisimple trace of the Frobenius $\Phi_p^r$ on
\[
H^0(\partial \mathcal{M}_{p^n m \overline{\mathbb{Q}}},\overline{\mathbb{Q}}_{\ell})
\]
is given by
\[
\frac 12 \int_{\mathrm{GL}_2(\hat{\mathbb{Z}})} \int_{\mathbb{A}_f} f(k^{-1}\left(\begin{array}{cc} 1 & 0 \\ 0 & p^r \end{array}\right)\left(\begin{array}{cc} 1 & u \\ 0 & 1 \end{array}\right)k) du dk\ .
\]
\end{cor}

\begin{rem} Here, for all $p^{\prime}$ we use the Haar measure on $\mathbb{Q}_{p^{\prime}}$ that gives $\mathbb{Z}_{p^{\prime}}$ measure $1$; in particular, the subgroup $\hat{\mathbb{Z}}$ of $\mathbb{A}_f$ gets measure 1.
\end{rem}

\begin{proof} Note that if
\[
f^p(k^{-1}\left(\begin{array}{cc} 1 & u \\ 0 & p^r \end{array}\right)k)\neq 0\ ,
\]
then $p^r\equiv 1\modd m$, so that the integral is identically zero if $p^r\not\equiv 1\modd m$. In fact, in this case, $\Phi_p^r$ has no fixed points on $\partial \mathcal{M}_{p^n m \overline{\mathbb{Q}}}$. So assume now that $p^r\equiv 1\modd m$. In that case, the inertia subgroup at $p$ groups the points of $\partial \mathcal{M}_{p^n m \overline{\mathbb{Q}}}$ into packets of size $p^{n-1}(p-1)$ on which $\Phi_p^r$ acts trivially. Therefore the semisimple trace of $\Phi_p^r$ is
\[
\frac 1{p^{n-1}(p-1)} \# (\{\pm \left(\begin{array}{cc} 1 & \ast \\ 0 & 1 \end{array}\right) \}\backslash \mathrm{GL}_2(\mathbb{Z}/p^n m\mathbb{Z}))\ .
\]
But
\[
f^p(k^{-1}\left(\begin{array}{cc} 1 & 0 \\ 0 & p^r \end{array}\right)\left(\begin{array}{cc} 1 & u \\ 0 & 1 \end{array}\right)k) = \# \mathrm{GL}_2(\mathbb{Z}/m\mathbb{Z}) \mathrm{vol}(\mathrm{GL}_2(\hat{\mathbb{Z}}^p))^{-1}
\]
if $u\equiv 0\modd m$ and is $0$ otherwise, so that
\[\begin{aligned}
\int_{\mathrm{GL}_2(\hat{\mathbb{Z}}^p)} \int_{\mathbb{A}_f^p} f^p(k^{-1}\left(\begin{array}{cc} 1 & 0 \\ 0 & p^r \end{array}\right)\left(\begin{array}{cc} 1 & u \\ 0 & 1 \end{array}\right)k) du dk\\
= \# (\{\left(\begin{array}{cc} 1 & \ast \\ 0 & 1 \end{array}\right) \}& \backslash \mathrm{GL}_2(\mathbb{Z}/m\mathbb{Z}))\ .
\end{aligned}\]
This reduces us to the statement
\[
\int_{\mathrm{GL}_2(\mathbb{Z}_p)} \int_{\mathbb{Q}_p} (f_{p,r}\ast e_{\Gamma(p^n)_{\mathbb{Q}_p}})(k^{-1}\left(\begin{array}{cc} 1 & u \\ 0 & p^r \end{array}\right)k) du dk = p^{2n}-p^{2n-2}\ .
\]
But note that the left hand side is the orbital integral of $f_{p,r}\ast e_{\Gamma(p^n)_{\mathbb{Q}_p}}$ for
\[
\gamma = \left(\begin{array}{cc} 1 & 0 \\ 0 & p^r \end{array}\right)\ ,
\]
because more generally for $\gamma_1\neq \gamma_2$ and any $h\in C_c^{\infty}(\mathrm{GL}_2(\mathbb{Q}_p))$
\[\begin{aligned}
&\ \int_{\mathrm{GL}_2(\mathbb{Z}_p)} \int_{\mathbb{Q}_p} h(k^{-1}\left(\begin{array}{cc} \gamma_1 & 0 \\ 0 & \gamma_2 \end{array}\right)\left(\begin{array}{cc} 1 & u \\ 0 & 1 \end{array}\right)k) du dk\\
=&\ |1-\frac{\gamma_2}{\gamma_1}|_p^{-1}\int_{\mathrm{GL}_2(\mathbb{Z}_p)} \int_{\mathbb{Q}_p} h(k^{-1}\left(\begin{array}{cc} 1 & -u \\ 0 & 1 \end{array}\right)\left(\begin{array}{cc} \gamma_1 & 0 \\ 0 & \gamma_2 \end{array}\right)\left(\begin{array}{cc} 1 & u \\ 0 & 1 \end{array}\right)k) du dk\\
=&\ |1-\frac{\gamma_2}{\gamma_1}|_p^{-1} \mathrm{vol}(T(\mathbb{Z}_p)) \int_{T(\mathbb{Q}_p)\backslash \mathrm{GL}_2(\mathbb{Q}_p)} h(g^{-1}\left(\begin{array}{cc} \gamma_1 & 0 \\ 0 & \gamma_2 \end{array}\right)g) dg\ ,
\end{aligned}\]
as
\[
\mathrm{GL}_2(\mathbb{Q}_p) = B(\mathbb{Q}_p)\mathrm{GL}_2(\mathbb{Z}_p)\ .
\]
Note that in our case, $|1-\frac{\gamma_2}{\gamma_1}|_p^{-1}=|1-p^r|_p^{-1}=1$. But the orbital integral of $f_{p,r}\ast e_{\Gamma(p^n)_{\mathbb{Q}_p}}$ equals the corresponding twisted orbital integral of $\phi_p\ast e_{\Gamma(p^n)_{\mathbb{Q}_{p^r}}}$ which was calculated in Lemma \ref{TORegularSplit}.
\end{proof}

\section{The Arthur-Selberg trace formula}

This section serves to give the special case of the Arthur-Selberg trace formula for $\mathrm{GL}_2$ that will be needed. We simply specialize the formula in \cite{Arthur} for the trace of Hecke operators on the $L^2$-cohomology of locally symmetric spaces.

Let
\[
H_{(2)}^i = \lim_{\longrightarrow} H_{(2)}^i(\mathcal{M}_{m}(\mathbb{C}),\mathbb{C})
\]
be the inverse limit of the $L^2$-cohomologies of the spaces $\mathcal{M}_m(\mathbb{C})$. It is a smooth, admissible representation of $\mathrm{GL}_2(\mathbb{A}_f)$. Again, we define the element
\[
H_{(2)}^{\ast} = \sum_{i=0}^2 (-1)^i H_{(2)}^i
\]
in the Grothendieck group of smooth admissible representations of $\mathrm{GL}_2(\mathbb{A}_f)$. Then let
\[
\mathcal{L}(h) = \tr(h|H_{(2)}^{\ast})
\]
for any $h\in C_c^{\infty}(\mathrm{GL}_2(\mathbb{A}_f))$.

Let $Z\subset T\subset B\subset \mathrm{GL}_2$ be the center, the diagonal torus and the upper triangular matrices. Recall that for any reductive group $G$ over $\mathbb{Q}$, $\overline{G}$ is an inner form of $G$ over $\mathbb{Q}$ that is anisotropic modulo center over $\mathbb{R}$. For $\gamma\in \mathrm{GL}_2$, let $G_{\gamma}$ be its centralizer. Finally, let
\[
T(\mathbb{Q})^{\prime} = \{\gamma = \left(\begin{array}{cc} \gamma_1 & 0 \\ 0 & \gamma_2 \end{array}\right) \mid \gamma_1\gamma_2>0, |\gamma_1|<|\gamma_2| \}\ .
\]
Here and in the following, absolute values always denote the real absolute value.

\begin{thm}\label{TraceFormula} For any $h\in C_c^{\infty}(\mathrm{GL}_2(\mathbb{A}_f))$, we have
\[\begin{aligned}
\frac 12 \mathcal{L}(h) =\\
- &\sum_{\gamma\in Z(\mathbb{Q})} \mathrm{vol}(\overline{\mathrm{GL}}_2(\mathbb{Q})\backslash \overline{\mathrm{GL}}_2(\mathbb{A}_f)) h(\gamma)\\
+ &\sum_{\substack{\gamma\in \mathrm{GL}_2(\mathbb{Q})\setminus Z(\mathbb{Q}) \\ \mathrm{\rm semisimple\ conj.\ class}\\ \mathrm{\rm with}\ \gamma_{\infty}\ \mathrm{\rm elliptic}}} \mathrm{vol}(\overline{G}_{\gamma}(\mathbb{Q})\backslash \overline{G}_{\gamma}(\mathbb{A}_f)) \int_{G_{\gamma}(\mathbb{A}_f)\backslash \mathrm{GL}_2(\mathbb{A}_f)} h(g^{-1}\gamma g) dg\\
+ &\frac 12 \sum_{\gamma\in T(\mathbb{Q})^{\prime}} \int_{\mathrm{GL}_2(\hat{\mathbb{Z}})} \int_{\mathbb{A}_f} h(k\gamma\left(\begin{array}{cc} 1 & u \\ 0 & 1 \end{array}\right) k^{-1}) du dk\\
+ &\frac 14 \sum_{\gamma\in Z(\mathbb{Q})} \int_{\mathrm{GL}_2(\hat{\mathbb{Z}})} \int_{\mathbb{A}_f} h(k\gamma \left(\begin{array}{cc} 1 & u \\ 0 & 1 \end{array}\right) k^{-1}) du dk\ .
\end{aligned}\]
\end{thm}

\begin{proof} Specialize Theorem 6.1 of \cite{Arthur} to this case. In this proof, we will use freely the notation from that article. As a preparation, we note the following formula for a discrete series character.

\begin{lem}\label{LemmaDiscreteSeries} Let $\pi$ be the admissible representation of $\mathrm{GL}_2(\mathbb{R})$ given by the space of $O(2)$-finite functions on $\mathbb{P}^1(\mathbb{R})$ modulo the constant functions. Let $\Theta_{\pi}$ be its character. Then for regular elliptic $\gamma\in \mathrm{GL}_2(\mathbb{R})$
\[
\Theta_{\pi}(\gamma)=-1\ ,
\]
and for $\gamma\in T(\mathbb{Q})$ regular with diagonal entries $\gamma_1$, $\gamma_2$, one has
\[
\Theta_{\pi}(\gamma)=0
\]
if $\gamma_1\gamma_2 < 0$, whereas if $\gamma_1\gamma_2>0$, then
\[
\Theta_{\pi}(\gamma) = 2\frac{\min\{|\frac{\gamma_2}{\gamma_1}|,|\frac{\gamma_1}{\gamma_2}|\}^{\frac 12}}{||\frac{\gamma_2}{\gamma_1}|^{\frac 12}-|\frac{\gamma_1}{\gamma_2}|^{\frac 12}|}\ .
\]
\end{lem}

\begin{proof} This directly follows from the formula for induced characters.
\end{proof}

We remark that the representation $\pi$ from the lemma is the unique discrete series representation with trivial central and infinitesimal character.

Now we begin to analyze the formula of Theorem 6.1 in \cite{Arthur}. We claim that the term $\frac 12 \mathcal{L}(h)$ is equal to $\mathcal{L}_{\mu}(h)$ for $\mu = 1$ in Arthur's notation. For this, we recall that the $\mathbb{C}$-valued points of $\mathcal{M}_m$ are given by
\[
\mathcal{M}_m(\mathbb{C}) = \mathrm{GL}_2(\mathbb{Q}) \backslash \mathrm{GL}_2(\mathbb{A}) / (\mathrm{SO}_2(\mathbb{R})\times K_m)\ ,
\]
where
\[
K_m = \{ g\in \mathrm{GL}_2(\hat{\mathbb{Z}}) \mid g\equiv 1\modd m \}\ .
\]
Because $\mathrm{SO}_2(\mathbb{R})$ has index $2$ in a maximal compact subgroup, we get a factor of $2$ by Remark 3 after Theorem 6.1 in \cite{Arthur}, giving $\mathcal{L}(h) = 2 \mathcal{L}_{\mu}(h)$ for $\mu = 1$, as claimed.

The outer sum in Theorem 6.1 of \cite{Arthur} runs over $M=\mathrm{GL}_2$, $M=T$.

Consider first the summand for $M=\mathrm{GL}_2$. The factor in front of the inner sum becomes $1$. The inner sum runs over semisimple conjugacy classes in $\gamma\in \mathrm{GL}_2(\mathbb{Q})$ which are elliptic at $\infty$. For the first factor $\chi(G_{\gamma})$, we use Remark 2 after Theorem 6.1 in \cite{Arthur}, noting that the term called $|\mathcal{D}(G,B)|$ in that formula is equal to $1$, and the sign is $-1$ or $+1$ corresponding to (in this order) $\gamma$ being central or not. The second factor $|\iota^{\mathrm{GL}_2}(\gamma)|$ equals $1$ in all cases, because all centralizers $G_{\gamma}$ are connected (as algebraic groups over $\mathbb{Q}$), cf. equation (6.1) of \cite{Arthur}. Now, by Lemma \ref{LemmaDiscreteSeries} and the definition of $\Phi_M(\gamma,\mu)$ (equation (4.4) of \cite{Arthur}), $\Phi_{\mathrm{GL}_2}(\gamma,\mu) = 1$ for regular elliptic $\gamma$ and hence all elliptic $\gamma$, by the way that $\Phi_{\mathrm{GL}_2}(\gamma,\mu)$ is extended, cf. p. 275 of \cite{Arthur}. Finally, the term $h_{\mathrm{GL}_2}(\gamma)$ is precisely the orbital integral $O_h(\gamma)$, by equation (6.2) of \cite{Arthur}. This takes care of all the terms for $M=\mathrm{GL}_2$, giving the first two summands in our formula.

Now, consider the case $M=T$. The factor in front of the inner sum becomes $-\frac 12$. The inner sum runs over the elements $\gamma\in T(\mathbb{Q})$. Let the diagonal entries of $\gamma$ be $\gamma_1$, $\gamma_2$. To evaluate $\chi(T_{\gamma})=\chi(T)$, we use Remark 2 after Theorem 6.1 in \cite{Arthur} again, to get $\chi(T_{\gamma})=\mathrm{vol}(T(\mathbb{Q})\backslash T(\mathbb{A}_f)) = \frac 14$. The term $|\iota^T(\gamma)|$ gives $1$, by the same reasoning as above. We want to evaluate the term $\Phi_M(\gamma,\mu)=\Phi_M(\gamma,1)$. Consider first the case of regular $\gamma$. By Lemma \ref{LemmaDiscreteSeries} and the definition (4.4) in \cite{Arthur}, we get that $\Phi_T(\gamma,\mu)$ is $0$ if $\gamma_1\gamma_2 < 0$, and otherwise equal to $-2\min\{|\frac{\gamma_2}{\gamma_1}|,|\frac{\gamma_1}{\gamma_2}|\}^{\frac 12}$. The same reasoning as above shows that this result continues to hold for non-regular $\gamma$. If $|\gamma_1|\leq |\gamma_2|$, then the fourth factor $h_T(\gamma)$ appears in the form of equation (6.3) of \cite{Arthur} in our formula, noting that $\delta_B(\gamma_{\mathrm{fin}})^{\frac 12} = |\frac{\gamma_2}{\gamma_1}|^{\frac 12}$. Finally, note that exchanging $\gamma_1$ and $\gamma_2$ does not change $h_{T}(\gamma)$, so that we may combine those terms. This gives the desired formula.
\end{proof}

We shall also need the spectral expansion for $\mathcal{L}(h)$. Let $\Pi_{\mathrm{disc}}(\mathrm{GL}_2(\mathbb{A}),1)$ denote the set of irreducible automorphic representations $\pi=\bigotimes_{p\leq \infty} \pi_p$ of $\mathrm{GL}_2(\mathbb{A})$ with $\pi_{\infty}$ having trivial central and infinitesimal character, that occur discretely in
\[
L^2(\mathrm{GL}_2(\mathbb{Q})\mathbb{R}_{>0}\backslash \mathrm{GL}_2(\mathbb{A}))\ .
\]
For $\pi\in \Pi_{\mathrm{disc}}(\mathrm{GL}_2(\mathbb{A}),1)$, let $m(\pi)$ be the multiplicity of $\pi$ in $L^2(\mathrm{GL}_2(\mathbb{Q})\mathbb{R}_{>0}\backslash \mathrm{GL}_2(\mathbb{A}))$. Using the relative Lie algebra cohomology groups, we have the following lemma.

\begin{lem}\label{L2Expansion} For any $i=0,1,2$, there is a canonical $\mathrm{GL}_2(\mathbb{A}_f)$-equivariant isomorphism
\[
H_{(2)}^i \cong \bigoplus_{\pi\in \Pi_{\mathrm{disc}}(\mathrm{GL}_2(\mathbb{A}),1)} m(\pi) H^i(\mathfrak{gl}_2,\mathrm{SO}_2(\mathbb{R}),\pi_{\infty}) \pi_f\ .
\]
Furthermore,
\[
H^i(\mathfrak{gl}_2,\mathrm{SO}_2(\mathbb{R}),\pi_{\infty}) = 0
\]
for all $i=0,1,2$ except if $\pi_{\infty}$ has trivial central and infinitesimal character. This gives the following cases:

\begin{enumerate}
\item[{\rm (i)}] $\pi_{\infty}$ is the trivial representation or $\pi_{\infty}=\mathrm{sgn}\ \det$. Then
\[
H^i(\mathfrak{gl}_2,\mathrm{SO}_2(\mathbb{R}),\pi_{\infty}) = \left\{ \begin{array}{cl} \mathbb{C} & i = 0\\ 0 & i = 1\\ \mathbb{C} & i = 2\ ; \end{array}\right .
\]

\item[{\rm (ii)}] $\pi_{\infty}$ is the representation from Lemma \ref{LemmaDiscreteSeries}. Then
\[
H^i(\mathfrak{gl}_2,\mathrm{SO}_2(\mathbb{R}),\pi_{\infty}) = \left\{ \begin{array}{cl} 0 & i = 0\\ \mathbb{C}\oplus\mathbb{C} & i = 1\\ 0 & i = 2\ . \end{array}\right .
\]
\end{enumerate}
\end{lem}

\begin{proof} The first part is taken from the discussion in Section 2 of \cite{Arthur}. The second part is contained in \cite{Casselman}.
\end{proof}

Denote $\chi(\pi_{\infty}) = \sum_{i=0}^2 (-1)^i \dim H^i(\mathfrak{gl}_2,\mathrm{SO}_2(\mathbb{R}),\pi_{\infty})$. We get the following corollary.

\begin{cor}\label{SpectralExpansion} For any $h\in C_c^{\infty}(\mathrm{GL}_2(\mathbb{A}_f))$, we have
\[
\mathcal{L}(h) = \sum_{\pi\in \Pi_{\mathrm{disc}}(\mathrm{GL}_2(\mathbb{A}),1)} m(\pi) \chi(\pi_{\infty}) \tr(h|\pi_f)\ .
\]
\end{cor}

\section{Comparison of the Lefschetz and Arthur-Selberg trace formula}

We deduce the following theorem. For this, we need to fix an isomorphism $\overline{\mathbb{Q}}_{\ell}\cong \mathbb{C}$.

\begin{thm}\label{CompSemisimpleTrace} With $f$ as above,
\[
2 \mathrm{tr}^{\mathrm{ss}}(\Phi_p^r | H^{\ast}(\overline{\mathcal{M}}_{p^n m \overline{\mathbb{Q}}},\overline{\mathbb{Q}}_{\ell})) = \mathcal{L}(f)\ .
\]
\end{thm}

\begin{proof} We compare the formulas given by Theorem \ref{TraceFinite}, Corollary \ref{TraceInfinity} and Theorem \ref{TraceFormula}. We are left to show that whenever
\[
\gamma = \left(\begin{array}{cc} \gamma_1 & 0 \\ 0 & \gamma_2 \end{array}\right)\in T(\mathbb{Q})
\]
with $\gamma_1\gamma_2 > 0$ and $|\gamma_1|\leq |\gamma_2|$, then
\[
\int_{\mathrm{GL}_2(\hat{\mathbb{Z}})} \int_{\mathbb{A}_f} f(k^{-1}\gamma \left(\begin{array}{cc} 1 & u \\ 0 & 1 \end{array}\right) k) du dk = 0
\]
except for $\gamma_1 = 1$, $\gamma_2 = p^r$.

However, the integral factors into a product of local integrals and the integral for a prime $\ell \neq p$ is only nonzero if $\gamma\in \mathrm{GL}_2(\mathbb{Z}_{\ell})$. It follows that $\gamma_1$ and $\gamma_2$ are up to sign a power of $p$.

Next, we claim that
\[
\int_{\mathrm{GL}_2(\mathbb{Z}_p)} \int_{\mathbb{Q}_p} (f_{p,r}\ast e_{\Gamma(p^n)_{\mathbb{Q}_p}})(k^{-1}\gamma \left(\begin{array}{cc} 1 & u \\ 0 & 1 \end{array}\right) k) du dk\neq 0
\]
only if $v_p(\gamma_1)=0$ and $v_p(\gamma_2)=r$, or the other way around. Indeed, as long as $\gamma_1\neq \gamma_2$, the term is up to a constant an orbital integral of $f_{p,r}\ast e_{\Gamma(p^n)_{\mathbb{Q}_p}}$, cf. proof of Corollary \ref{TraceInfinity}, and we have computed those, by computing the twisted orbital integrals of the matching function $\phi_p\ast e_{\Gamma(p^n)_{\mathbb{Q}_{p^r}}}$. The case $\gamma_1 = \gamma_2$ follows by continuity of the integrals.

As $\gamma_1$ and $\gamma_2$ are up to sign powers of $p$, we are left with either $\gamma_1 = 1$ and $\gamma_2 = p^r$ or $\gamma_1 = -1$ and $\gamma_2 = -p^r$. But the second case also gives $0$, because no conjugate of $\gamma$ will be $\equiv 1\modd m$.
\end{proof}

This finally allows us to compute the zeta-function of the varieties $\overline{\mathcal{M}}_{m}$. Here, $m$ is any integer which is the product of two coprime integers, both at least $3$, and we do not consider any distinguished prime. Recall that the Hasse-Weil zeta-function of a variety $X$ over a number field $K$ is defined as a product of the local factors,
\[
\zeta(X,s)=\prod_{\lambda} \zeta(X_{K_{\lambda}},s)\ ,
\]
convergent for all complex numbers $s$ whose real part is large enough. Here $\lambda$ runs through the finite places of $K$ and $X_{K_\lambda}$ denotes the base-change of $X$ to the local field $K_\lambda$.

\begin{thm} The Hasse-Weil zeta-function of $\overline{\mathcal{M}}_{m}$ is given by
\[
\zeta(\overline{\mathcal{M}}_{m},s) = \prod_{\pi\in \Pi_{\mathrm{disc}}(\mathrm{GL}_2(\mathbb{A}),1)} L(\pi,s-\tfrac 12)^{\frac 12 m(\pi) \chi(\pi_{\infty}) \dim \pi_f^{K_m}}\ ,
\]
where
\[
K_m = \{g\in \mathrm{GL}_2(\hat{\mathbb{Z}}) \mid g\equiv 1\modd m\}\ .
\]
\end{thm}

\begin{proof} We compute the semisimple local factors at all primes $p$. For this, write $m=p^n m^{\prime}$, where $m^{\prime}$ is not divisible by $p$. By assumption on $m$, we get $m^{\prime}\geq 3$. Combining Theorem \ref{CompSemisimpleTrace} and Corollary \ref{SpectralExpansion}, one sees that
\begin{equation}\label{eqn012}
\begin{array}{c}
\displaystyle \sum_{i=0}^2 (-1)^i \mathrm{tr}^{\mathrm{ss}}(\Phi_p^r | H^i(\overline{\mathcal{M}}_{m,\overline{\mathbb{Q}}_p},\overline{\mathbb{Q}}_{\ell}) )\\
\displaystyle = \frac 12 p^{\frac 12 r} \sum_{\pi\in \Pi_{\mathrm{disc}}(\mathrm{GL}_2(\mathbb{A}),1)} m(\pi) \chi(\pi_{\infty}) \mathrm{tr}^{\mathrm{ss}}(\Phi_p^r | \sigma_{\pi_p}) \dim \pi_f^{K_m}\ .
\end{array}
\end{equation}

We check by hand that also
\begin{equation}\label{eqn02}
\begin{array}{c}
\displaystyle \sum_{i\in \{0,2\}} (-1)^i \mathrm{tr}^{\mathrm{ss}}(\Phi_p^r | H^i(\overline{\mathcal{M}}_{m,\overline{\mathbb{Q}}_p},\overline{\mathbb{Q}}_{\ell}) )\\
\displaystyle = \frac 12 p^{\frac 12 r} \sum_{\substack{\pi\in \Pi_{\mathrm{disc}}(\mathrm{GL}_2(\mathbb{A}),1)\\ \dim \pi_{\infty} = 1}} m(\pi) \chi(\pi_{\infty}) \mathrm{tr}^{\mathrm{ss}}(\Phi_p^r | \sigma_{\pi_p}) \dim \pi_f^{K_m}\ .
\end{array}
\end{equation}
Indeed, the sum on the right hand-side gives non-zero terms only for 1-dimensional representations $\pi$ which are trivial on $K_m$. Using $\chi(\pi_{\infty})=2$, $\dim \pi_f^{K_m}=1$ and $m(\pi)=1$, the statement then reduces to class field theory, as the geometric connected components of $\overline{\mathcal{M}}_m$ are parametrized by the primitive $m$-th roots of unity. Note that in \eqref{eqn02} one may replace the semisimple trace by the usual trace on the $I_{\mathbb{Q}_p}$-invariants everywhere. This gives
\begin{equation}\label{eqn02good}
\begin{array}{c}
\displaystyle \prod_{i\in \{0,2\}} \det(1-\Phi_p p^{-s} | H^i(\overline{\mathcal{M}}_{m,\overline{\mathbb{Q}}_p},\overline{\mathbb{Q}}_{\ell})^{I_{\mathbb{Q}_p}} ) \\
\displaystyle = \prod_{\substack{\pi\in \Pi_{\mathrm{disc}}(\mathrm{GL}_2(\mathbb{A}),1)\\ \dim \pi_{\infty} = 1}} L(\pi_p,s-\tfrac 12)^{\frac 12 m(\pi) \chi(\pi_{\infty}) \dim \pi_f^{K_m}}\ .
\end{array}
\end{equation}

Subtracting \eqref{eqn02} from \eqref{eqn012}, we see that
\begin{equation}
\begin{array}{c}
\displaystyle - \mathrm{tr}^{\mathrm{ss}}(\Phi_p^r | H^1(\overline{\mathcal{M}}_{m,\overline{\mathbb{Q}}_p},\overline{\mathbb{Q}}_{\ell}) ) \\
\displaystyle = \frac 12 p^{\frac 12 r} \sum_{\substack{\pi\in \Pi_{\mathrm{disc}}(\mathrm{GL}_2(\mathbb{A}),1)\\ \dim \pi_{\infty} > 1}} m(\pi) \chi(\pi_{\infty}) \mathrm{tr}^{\mathrm{ss}}(\Phi_p^r | \sigma_{\pi_p}) \dim \pi_f^{K_m}\ ,
\end{array}
\end{equation}
or equivalently
\begin{equation}
\begin{array}{c}
\displaystyle \mathrm{det}^{\mathrm{ss}}(1-\Phi_p p^{-s} | H^1(\overline{\mathcal{M}}_{m,\overline{\mathbb{Q}}_p},\overline{\mathbb{Q}}_{\ell}) )^{-1} \\
\displaystyle = \prod_{\substack{\pi\in \Pi_{\mathrm{disc}}(\mathrm{GL}_2(\mathbb{A}),1)\\ \dim \pi_{\infty} > 1}} L(\sigma_{\pi_p}^{\mathrm{ss}},s-\tfrac 12)^{\frac 12 m(\pi) \chi(\pi_{\infty}) \dim \pi_f^{K_m}}\ ,
\end{array}
\end{equation}
with the obvious definition for the semisimple determinant. All zeroes of the left-hand side have imaginary part $0$, $\frac 12$ or $1$: Indeed, if $\overline{\mathcal{M}}_{m,\mathbb{Q}_p}$ had good reduction, the Weil conjectures would imply that all zeroes have imaginary part $\frac 12$. In general, the semistable reduction theorem for curves together with the Rapoport-Zink spectral sequence imply that all zeroes have imaginary part $0$, $\frac 12$ or $1$. Changing the semisimple determinant to the usual determinant on the invariants under $I_{\mathbb{Q}_p}$ exactly eliminates the zeroes of imaginary part $1$, by the monodromy conjecture, proven in dimension 1 in \cite{RapoportZink}.

We also see that all zeroes of the right-hand side have imaginary part $0$, $\frac 12$ or $1$. Assume $\pi$ gives a nontrivial contribution to the right-hand side. Then $\pi_p$ cannot be 1-dimensional, because otherwise $\pi$ and hence $\pi_{\infty}$ would be 1-dimensional. Hence $\pi_p$ is generic. Being also unitary, the $L$-factor $L(\pi_p,s-\frac 12)$ of $\pi_p$ cannot have poles with imaginary part $\geq 1$, so that replacing $L(\sigma_{\pi_p}^{\mathrm{ss}},s-\frac 12)$ by $L(\pi_p,s-\frac 12)$ consists again in removing all zeroes of imaginary part $1$.

We find that
\begin{equation}\label{eqn1good}
\begin{array}{c}
\displaystyle \det(1-\Phi_p p^{-s} | H^1(\overline{\mathcal{M}}_{m,\overline{\mathbb{Q}}_p},\overline{\mathbb{Q}}_{\ell})^{I_{\mathbb{Q}_p}} )^{-1} \\
\displaystyle = \prod_{\substack{\pi\in \Pi_{\mathrm{disc}}(\mathrm{GL}_2(\mathbb{A}),1)\\ \dim \pi_{\infty} > 1}} L(\pi_p,s-\tfrac 12)^{\frac 12 m(\pi) \chi(\pi_{\infty}) \dim \pi_f^{K_m}}\ .
\end{array}
\end{equation}

Combining \eqref{eqn02good} and \eqref{eqn1good} yields the result.
\end{proof}

\section{Explicit determination of $\phi_p\ast e_{\Gamma(p^n)_{\mathbb{Q}_{p^r}}}$}

In this section, we aim to determine the values of the function $\phi_p\ast e_{\Gamma(p^n)_{\mathbb{Q}_{p^r}}}$ for $n\geq 1$. Set $q=p^r$.

For any $g\in \mathrm{GL}_2(\mathbb{Q}_q)$, we let $k(g)$ denote the minimal number $k$ such that $p^kg$ has integral entries. If additionally, $v_p(\det g)\geq 1$ and $v_p(\tr g)=0$, then $g$ has a unique eigenvalue $x\in \mathbb{Q}_q$ with $v_p(x)=0$; we define $\ell(g)=v_p(x-1)$ in this case. The choice of the maximal compact subgroup $\mathrm{GL}_2(\mathbb{Z}_q)$ gives a vertex $v_0$ in the building of $\mathrm{PGL}_2$. We will need another characterization of $k(g)$.

\begin{lem} For any $g\in \mathrm{GL}_2(\mathbb{Q}_q)$ which is conjugate to an integral matrix, consider the set $V_g$ of all vertices $v$ such that $g(\Lambda_v)\subset \Lambda_v$, where $\Lambda_v$ is the lattice corresponding to $v$. Then the distance of $v$ and $v_0$ is at least $k(g)$ for all $v\in V_g$ and there is a unique vertex $v(g)\in V_g$ such that the distance of $v(g)$ and $v_0$ is equal to $k(g)$.
\end{lem}

\begin{proof} Note that if $k(g)=0$, then this is trivial. So assume $k(g)>0$.

It is technically more convenient to use norms (or equivalently valuations) instead of lattices. So let $\mathrm{val}_v: \mathbb{Z}_q^2\longrightarrow \mathbb{R}$ be the valuation associated to $v$ in the building of $\mathrm{PGL}_2$; it is well-defined up to a constant. Then the distance of $v$ and $v_0$ is
\[
\max (\mathrm{val}_v(y) - \mathrm{val}_v(x) + \mathrm{val}_{v_0}(x) - \mathrm{val}_{v_0}(y))\ .
\]
Now by definition of $k(g)$, one has $\mathrm{val}_{v_0}(gx)\geq \mathrm{val}_{v_0}(x) - k(g)$ for all $x\in \mathbb{Z}_q^2$, but there is some $x\in \mathbb{Z}_q^2$ with $\mathrm{val}_{v_0}(gx) = \mathrm{val}_{v_0}(x) - k(g)$. Fix such an $x$ and set $y=gx$. Assuming that $v\in V_g$, we have $\mathrm{val}_v(gx)\geq \mathrm{val}_v(x)$, so that
\[
\mathrm{dist} (v,v_0)\geq \mathrm{val}_v(gx) - \mathrm{val}_v(x) + \mathrm{val}_{v_0}(x) - \mathrm{val}_{v_0}(gx)\geq k(g)\ ,
\]
giving the first claim. But we may more generally set $y=gx+ax$ for any $a\in \mathbb{Z}_q$. Then we still have $\mathrm{val}_{v_0}(y)\leq \mathrm{val}_{v_0}(x) - k(g)$ (since $k(g)>0$) and $\mathrm{val}_v(y)\geq \mathrm{val}_v(x)$, giving
\[
\mathrm{dist} (v,v_0)\geq \mathrm{val}_v(y) - \mathrm{val}_v(x) + \mathrm{val}_{v_0}(x) - \mathrm{val}_{v_0}(y)\geq k(g)\ ,
\]
as before. It follows that if $\mathrm{dist} (v,v_0)=k(g)$, then necessarily $\mathrm{val}_v(gx+ax)=\mathrm{val}_v(x)$ for all $a\in \mathbb{Z}_q$. But then it is clear that $\Lambda_v$ is the lattice generated by $gx$ and $x$, so that we get uniqueness.

Finally, define
\[
\mathrm{val}_g(x) = \min (\mathrm{val}_{v_0}(x),\mathrm{val}_{v_0}(gx) )\ .
\]
Since $g^2x=(\tr g) gx - (\det g) x$ and both $\tr g$ and $\det g$ are integral, one sees that $\mathrm{val}_g(gx)\geq \mathrm{val}_g(x)$. Furthermore,
\[
\mathrm{val}_g(y) - \mathrm{val}_g(x) + \mathrm{val}_{v_0}(x) - \mathrm{val}_{v_0}(y)\leq \min (0, \mathrm{val}_{v_0}(x) - \mathrm{val}_{v_0}(gx))\leq k(g)\ ,
\]
so that by what we have already proved, $\mathrm{val}_g(x)$ corresponds to a point $v(g)\in V_g$ with distance $k(g)$ to $v_0$.
\end{proof}

\begin{rem} This lemma is related to the fact that for all matrices $g\in \mathrm{GL}_2(\mathbb{Q}_q)$, the set $V_g$ is convex.
\end{rem}

For a fixed vertex $v$, the set of all $g\in \mathrm{GL}_2(\mathbb{Q}_q)$, conjugate to some integral matrix, with $v(g)=v$ is denoted $G_v$. By $\overline{G}_v$, we mean the set of all $g$ which map the lattice corresponding to $v$ into itself.

It is clear that if $v_p(\det g)\geq 1$ and $v_p(\tr g)=0$ (so that $\ell(g)$ is defined), then $\ell(g)=v_p(1 - \tr g + \det g)$.

For $n\geq 1$, we define a function $\phi_{p,n}: \mathrm{GL}_2(\mathbb{Q}_q)\longrightarrow \mathbb{C}$ by the following requirements:
\begin{itemize}
\item $\phi_{p,n}(g)=0$ except if $v_p(\det g)=1$, $v_p(\tr g)\geq 0$ and $k(g)\leq n-1$. Assume now that $g$ has these properties.
\item $\phi_{p,n}(g)=-1-q$ if $v_p(\tr g)\geq 1$,
\item $\phi_{p,n}(g)=1-q^{2\ell(g)}$ if $v_p(\tr g)=0$ and $\ell(g)<n-k(g)$,
\item $\phi_{p,n}(g)=1+q^{2(n-k(g))-1}$ if $v_p(\tr g)=0$ and $\ell(g)\geq n-k(g)$.
\end{itemize}

\begin{thm}\label{CalcFunction} Choose the Haar measure on $\mathrm{GL}_2(\mathbb{Q}_q)$ such that a maximal compact subgroup has measure $q-1$. Then
\[
\phi_{p,n} = \phi_p\ast e_{\Gamma(p^n)_{\mathbb{Q}_q}}\ .
\]
\end{thm}

\begin{proof} It is enough to check that $\phi_{p,n}$ lies in the center of the Hecke algebra and that the semisimple orbital integrals of $\phi_{p,n}$ and $\phi_p\ast e_{\Gamma(p^n)_{\mathbb{Q}_q}}$ agree.

In the case $q=p$, we have computed the orbital integrals of $\phi_p\ast e_{\Gamma(p^n)_{\mathbb{Q}_q}}$ in Theorem \ref{MainTheorem}. In fact, the calculation goes through for all powers $q$ of $p$. Recall that $\phi_{p,0}$ is the characteristic function of $\mathrm{GL}_2(\mathbb{Z}_q) \left(\begin{array}{cc} p & 0\\ 0 & 1\end{array}\right) \mathrm{GL}_2(\mathbb{Z}_q)$ divided by the volume of $\mathrm{GL}_2(\mathbb{Z}_q)$.

\begin{thm}\label{MainTheoremPrime} Let $\gamma\in \mathrm{GL}_2(\mathbb{Q}_q)$ be semisimple. Then
\[
O_{\gamma}(\phi_p\ast e_{\Gamma(p^n)_{\mathbb{Q}_q}}) = O_{\gamma}(\phi_{p,0}) c(\gamma)\ ,
\]
where
\[
c(\gamma)=\left\{\begin{array}{ll}
(1+q)(1-q^n) & v_p(\det \gamma)=1, v_p(\tr \gamma)\geq 1\\
q^{2n}-q^{2n-2} & v_p(\det \gamma)=1, v_p(\tr \gamma)=0, \ell(\gamma)\geq n\\
0 & \mathrm{else}\ . \end{array} \right .
\]
\end{thm}

\begin{proof} Note that when $q=p$, then $c(\gamma)=c_1(\gamma,e_{\Gamma(p^n)_{\mathbb{Q}_p}})$, so that this is Theorem \ref{MainTheorem}. But note that we never used that $p$ is a prime in the local harmonic analysis, so that replacing $\mathbb{Q}_p$ by $\mathbb{Q}_q$ everywhere gives the result for general $q$.
\end{proof}

We now aim at proving the same formula for $\phi_{p,n}$.

\begin{prop}\label{CompOrbitalIntegralPhi} Let $\gamma\in \mathrm{GL}_2(\mathbb{Q}_q)$ be semisimple. Then
\[
O_{\gamma}(\phi_{p,n}) = O_{\gamma}(\phi_{p,0}) c(\gamma)\ .
\]
\end{prop}

\begin{proof} First of all, note that $O_{\gamma}(\phi_{p,n})$ can only be nonzero if $\gamma$ is conjugate to an integral matrix and $v_p(\det \gamma)=1$. Of course, the same holds for $O_{\gamma}(\phi_{p,0})$. Hence we only need to consider the case that $\gamma$ is integral and $v_p(\det \gamma)=1$.

For any vertex $v$ of the building of $\mathrm{PGL}_2$, let
\[
G_{v,\gamma}=\{g\in \mathrm{GL}_2(\mathbb{Q}_q)\mid v(g^{-1}\gamma g)=v \}\ .
\]

\begin{lem}\label{SizeForEachVertex} For any $v\neq v_0$, we have
\[
\frac{\mathrm{vol}(G_{\gamma}(\mathbb{Q}_q)\backslash G_{v,\gamma})}{\mathrm{vol}(G_{\gamma}(\mathbb{Q}_q)\backslash G_{v_0,\gamma})} = \left\{\begin{array}{ll} \frac{q}{q+1} & \tr \gamma\equiv 0\modd p \\ \frac{q-1}{q+1} & \tr \gamma\not\equiv 0\modd p\ . \end{array}\right .
\]
\end{lem}

\begin{proof} Let $v^{\prime}$ be the first vertex on the path from $v$ to $v_0$. Then $G_v = \overline{G}_v\setminus \overline{G}_{v^{\prime}}$. Furthermore, $\overline{G}_v$ is conjugate to $\overline{G}_{v_0}$. Under this conjugation, $v^{\prime}$ is taken to some vertex $v_0^{\prime}$ that is a neighbor of $v_0$. In fact, one may choose $v_0^{\prime}$ arbitrarily. We see that $G_v$ is conjugate to
\[
G_{v_0}\setminus \overline{G}_{v_0^{\prime}}
\]
for all neighbors $v_0^{\prime}$ of $v_0$ (note that $G_{v_0}=\overline{G}_{v_0}$). Hence $G_{v,\gamma}$ is conjugate to the set
\[
\{g\in \mathrm{GL}_2(\mathbb{Q}_q)\mid g^{-1}\gamma g\in G_{v_0}\setminus \overline{G}_{v_0^{\prime}}\}
\]
which is obviously a subset of $G_{v_0,\gamma}$. We check that if $\gamma\in G_{v_0}$ and $\tr \gamma\equiv 0\modd p$, then there are $q$ (out of $q+1$) neighbors $v_0^{\prime}$ of $v_0$ such that $\gamma\not\in \overline{G}_{v_0^{\prime}}$ and if $\tr \gamma\not\equiv 0\modd p$, then there are $q-1$ neighbors with this property. In fact, $\gamma\in \overline{G}_{v_0^{\prime}}$ if and only if $\gamma\modd p$ stabilizes the line in $\mathbb{F}_q^2$ corresponding to $v_0^{\prime}$. Now in the first case, $\gamma\modd p$ has only eigenvalue $0$, with geometric multiplicity $1$, whereas in the second case, $\gamma\modd p$ has two distinct eigenvalues $0$ and $\tr \gamma$.

Using this with $g^{-1}\gamma g$ in place of $\gamma$, we see that each element of $G_{v_0,\gamma}$ lies in precisely $q$ (resp. $q-1$) of the $q+1$ sets
\[
\{g\in \mathrm{GL}_2(\mathbb{Q}_q)\mid g^{-1}\gamma g\in G_{v_0}\setminus \overline{G}_{v_0^{\prime}}\}
\]
indexed by $v_0^{\prime}$. This gives the claim.
\end{proof}

Note that we have (by our choice of Haar measure, giving a maximal compact subgroup measure $q-1$)
\[
O_{\gamma}(\phi_{p,0}) = \frac 1{q-1} \mathrm{vol}(G_{\gamma}(\mathbb{Q}_q)\backslash G_{v_0,\gamma})\ .
\]

Now we are reduced to a simple counting argument. Assume first that $\tr \gamma\equiv 0\modd p$. Then $\phi_{p,n}(g^{-1}\gamma g)=-1-q$ as long as $k(g^{-1}\gamma g)\leq n-1$; otherwise, it gives $0$. There are $(q+1)(1+q+q^2+...+q^{n-2})$ vertices $v\neq v_0$ with distance at most $n-1$. Hence, by the Lemma,
\[\begin{aligned}
O_{\gamma}(\phi_{p,n}) &= -(1+q) \mathrm{vol}(G_{\gamma}(\mathbb{Q}_q)\backslash G_{v_0,\gamma}) \\
&- (1+q)(q+1)(1+q+...+q^{n-2}) \frac{q}{q+1}\mathrm{vol}(G_{\gamma}(\mathbb{Q}_q)\backslash G_{v_0,\gamma}) \\
&= - (1+q)(1+q+...+q^{n-1})\mathrm{vol}(G_{\gamma}(\mathbb{Q}_q)\backslash G_{v_0,\gamma})\ .
\end{aligned}\]
Comparing, we get the claim.

Now assume that $\tr \gamma\not\equiv 0\modd p$, so that $\ell(\gamma)$ is defined. Assume that $\ell(\gamma)<n$. Then for $k(g^{-1}\gamma g)<n-\ell(g)$, we have
\[
\phi_{p,n}(g^{-1}\gamma g)=1-q^{2\ell(g)}\ ,
\]
for $n-\ell(g)\leq k(g^{-1}\gamma g) < n$, we have
\[
\phi_{p,n}(g^{-1}\gamma g)=1+q^{2(n-k(g^{-1}\gamma g))-1}\ ,
\]
and in all other cases we get $0$. Therefore, again by the Lemma,
\[\begin{aligned}
O_{\gamma}(\phi_{p,n}) &= (1-q^{2\ell(g)}) \mathrm{vol}(G_{\gamma}(\mathbb{Q}_q)\backslash G_{v_0,\gamma}) \\
&+ (1-q^{2\ell(g)})(q+1)(1+q+...+q^{n-\ell(g)-2}) \frac{q-1}{q+1} \mathrm{vol}(G_{\gamma}(\mathbb{Q}_q)\backslash G_{v_0,\gamma}) \\
&+ \Big((1+q^{2\ell(g)-1})(q+1)q^{n-\ell(g)-1}+...+(1+q^3)(q+1)q^{n-3} \\
&+ (1+q)(q+1)q^{n-2}\Big) \frac{q-1}{q+1} \mathrm{vol}(G_{\gamma}(\mathbb{Q}_q)\backslash G_{v_0,\gamma}) \\
&= \Big((1-q^{2\ell(g)})q^{n-\ell(g)-1} + (q-1)(q^{n-\ell(g)-1}+...+q^{n-3}+q^{n-2}\\
&+ q^{n-1}+q^n+...+q^{n+\ell(g)-2})\Big) \mathrm{vol}(G_{\gamma}(\mathbb{Q}_q)\backslash G_{v_0,\gamma}) \\
&= ((1-q^{2\ell(g)})q^{n-\ell(g)-1} + (q^{2\ell(g)}-1)q^{n-\ell(g)-1}) \mathrm{vol}(G_{\gamma}(\mathbb{Q}_q)\backslash G_{v_0,\gamma}) \\
&= 0\ ,
\end{aligned}\]
as claimed.

Finally, assume $\ell(\gamma)=n$. Then $\phi_{p,n}(g^{-1}\gamma g)=1+q^{2(n-k(g^{-1}\gamma g))-1}$ if $k(g^{-1}\gamma g)<n$ and vanishes otherwise. This shows that
\[\begin{aligned}
O_{\gamma}(\phi_{p,n}) &= (1+q^{2n-1}) \mathrm{vol}(G_{\gamma}(\mathbb{Q}_q)\backslash G_{v_0,\gamma}) \\
&+ \Big((1+q^{2n-3})(q+1)q^0+...+(1+q^3)(q+1)q^{n-3} \\
&+(1+q)(q+1)q^{n-2}\Big)\frac{q-1}{q+1} \mathrm{vol}(G_{\gamma}(\mathbb{Q}_q)\backslash G_{v_0,\gamma}) \\
&= \Big((1+q^{2n-1})+(q-1)(1+q+...+q^{n-3}+q^{n-2}\\
&+ q^{n-1}+q^n+...+q^{2n-3})\Big) \mathrm{vol}(G_{\gamma}(\mathbb{Q}_q)\backslash G_{v_0,\gamma}) \\
&= (q^{2n-1}+q^{2n-2}) \mathrm{vol}(G_{\gamma}(\mathbb{Q}_q)\backslash G_{v_0,\gamma})\ .
\end{aligned}\]
Again, this is what we have asserted.
\end{proof}

It remains to see that $\phi_{p,n}$ lies in the center of $\mathcal{H}(\mathrm{GL}_2(\mathbb{Q}_q), \Gamma(p^n)_{\mathbb{Q}_q})$. Our argument will be slightly indirect, as the direct approach would run into some convergence issues. We consider the following deformation $\phi_{p,n,t}$ of $\phi_{p,n}$:
\begin{itemize}
\item $\phi_{p,n,t}(g)=0$ except if $v_p(\det g)=1$, $v_p(\tr g)\geq 0$ and $k(g)\leq n-1$. Assume now that $g$ has these properties.
\item $\phi_{p,n,t}(g)=-q\frac{1-t^2}{q-t^2}$ if $v_p(\tr g)\geq 1$,
\item $\phi_{p,n,t}(g)=1-t^{2\ell(g)}$ if $v_p(\tr g)=0$ and $\ell(g)<n-k(g)$,
\item $\phi_{p,n,t}(g)=1-\frac{(q-1)t^{2(n-k(g))}}{q-t^2}$ if $v_p(\tr g)=0$ and $\ell(g)\geq n-k(g)$.
\end{itemize}

Then specializing to $t=q$, we have $\phi_{p,n,q}=\phi_{p,n}$. We claim that for all $t$, $\phi_{p,n,t}$ lies in the center of $\mathcal{H}(\mathrm{GL}_2(\mathbb{Q}_q),\Gamma(p^n)_{\mathbb{Q}_q})$. Since as a function of $t$, $\phi_{p,n,t}$ is a rational function and hence any identity of the form $\phi_{p,n,t}\ast f = f\ast \phi_{p,n,t}$ reduces to a polynomial identity in $t$, it suffices to check this for $t$ infinitesimally small. Hence consider $\phi_{p,n,t}(g)$ as a function on $\mathrm{GL}_2(\mathbb{Q}_q)$ with values in $\mathbb{C}[[t]]$. First we check that these functions are compatible for varying $n$.

\begin{prop} We have
\[
\phi_{p,n,t} = \phi_{p,n+1,t}\ast e_{\Gamma(p^n)_{\mathbb{Q}_q}}\ .
\]
\end{prop}

\begin{proof} We compare function values at $g\in \mathrm{GL}_2(\mathbb{Q}_q)$. We may assume that $v_p(\det g)=1$, as otherwise both functions vanish. Also if $k(g)\geq n+1$, then so is $k(gu)\geq n+1$ for all $u\in \Gamma(p^n)_{\mathbb{Q}_q}$, so that both sides give $0$. Hence we may assume $k(g)\leq n$. Note that in all cases, $k(gu)=k(g)$ for all $u\in \Gamma(p^n)_{\mathbb{Q}_q}$. Also recall that if $\ell(g)$ is defined, then $\ell(g)=v_p(1-\tr g+\det g)$.

Consider the case $k(g)=n$. We need to check that the right hand side gives $0$. Note that in this case, the value $\phi_{p,n+1,t}(gu)$ depends only on $\tr (gu)\modd p$. It is easy to see that each value of $\tr (gu)\modd p$ is taken the same number of times. For $\tr (gu)=0\modd p$, we have
\[
\phi_{p,n+1,t}(gu)=-q\frac{1-t^2}{q-t^2}\ ,
\]
for $\tr (gu)=1\modd p$, we have
\[
\phi_{p,n+1,t}(gu)=1-\frac{(q-1)t^2}{q-t^2}=q\frac{1-t^2}{q-t^2}\ ,
\]
and for all other values of $\tr (gu)\modd p$, we have $\phi_{p,n+1,t}(gu)=0$. This gives the result.

Now we can assume that $k(g)\leq n-1$. Assume first that $\tr g\equiv 0\modd p$. Then $\tr (gu) \equiv 0 \modd p$ and hence $\phi_{p,n+1,t}(gu) = \phi_{p,n,t}(g)$ for all $u\in \Gamma(p^n)_{\mathbb{Q}_q}$, giving the claim in this case.

We are left with $\tr g\not\equiv 0\modd p$. If $k(g)+\ell(g)<n$, then $k(gu)=k(g)$ and $\ell(gu)=\ell(g)$ for all $u\in \Gamma(p^n)_{\mathbb{Q}_q}$ and in particular again $k(gu)+\ell(gu)<n$, so that by definition $\phi_{p,n+1,t}(gu)=\phi_{p,n,t}(g)$.

So finally we are in the case $\tr g\not\equiv 0\modd p$, $k(g)+\ell(g)\geq n$, but $k(g)\leq n-1$. Then
\[
\phi_{p,n,t}(g)=1-\frac{(q-1)t^{2(n-k(g))}}{q-t^2}\ .
\]
We know that $\tr (gu)\equiv 1\modd p^{n-k(g)}$, but all values of $\tr(gu)\modd p^{n-k(g)+1}$ with this restriction are taken equally often. If
\[
\tr(gu)\equiv 1+\det g\modd p^{n+1-k(g)}\ ,
\]
then $\ell(gu)\geq n+1-k(g)$ (since $\det (gu)\equiv \det g\modd p^{n+1}$), so that
\[
\phi_{p,n+1,t}(gu) = 1-\frac{(q-1)t^{2(n+1-k(g))}}{q-t^2}\ .
\]
In all other cases, we have $\ell(gu)=n-k(g)$, so that
\[
\phi_{p,n+1,t}(gu) = 1-t^{2(n-k(g))}\ .
\]
Hence we get
\[\begin{aligned}
(\phi_{p,n+1,t}\ast e_{\Gamma(p^n)_{\mathbb{Q}_q}})(g) &= \frac 1q \left(1-\frac{(q-1)t^{2(n+1-k(g))}}{q-t^2}\right) \\
&+ \frac{q-1}q \left(1-t^{2(n-k(g))}\right) \\
&= 1-\frac{q-1}q t^{2(n-k(g))}\left(\frac{t^2}{q-t^2}+1\right) \\
&= 1-\frac{(q-1)t^{2(n-k(g))}}{q-t^2} = \phi_{p,n,t}(g)\ ,
\end{aligned}\]
as claimed.
\end{proof}

Hence we may consider the system $\phi_{p,t}=(\phi_{p,n,t})_n$ as a distribution with values in $\mathbb{C}[[t]]$ on the compactly supported, locally constant functions on $\mathrm{GL}_2(\mathbb{Q}_q)$ with the property that $\phi_{p,t}\ast e_K$ is compactly supported for all compact open subgroups $K$. To check that $\phi_{p,n,t}$ is central for all $n$, it remains to see that $\phi_{p,t}$ is conjugation-invariant. But note that $\phi_{p,t}\modd t^m$ is represented by a locally constant function for all $m$ -- the important point here is that $\phi_{p,t}\modd t^m$ becomes constant when one eigenvalue of $g$ approaches $1$. Here we need our deformation parameter $t$. It is also clear that $\phi_{p,t}\modd t^m$ is conjugation-invariant for all $m$, which finishes the proof.
\end{proof}

\bibliographystyle{abbrv}
\bibliography{ModularCurve}

\end{document}